\documentclass[]{amsart}
\usepackage[all]{xy}
\usepackage{graphicx}
\usepackage{amscd,amsthm,amssymb,amsfonts,amsmath,euscript}

\theoremstyle{plain}
\newtheorem{thm}{Theorem}[section]
\newtheorem{lm}[thm]{Lemma}
\newtheorem{prop}[thm]{Proposition}
\newtheorem{cor}[thm]{Corollary}

\newtheorem{fact}[thm]{Fact}

\theoremstyle{definition}
\newtheorem{defn}[thm]{Definition}
\newtheorem{eg}[thm]{Example}

\theoremstyle{remark}
\newtheorem{remark}[thm]{Remark}
\newtheorem*{thank}{Acknowledgments}
\newcommand{\nc}{\newcommand}

\def\makeop#1{\expandafter\def\csname#1\endcsname
  {\mathop{\rm #1}\nolimits}\ignorespaces}
\makeop{Hom}   \makeop{End}   \makeop{Aut}   \makeop{Isom}  \makeop{Pic} 
\makeop{Gal}   \makeop{ord}   \makeop{Char}  \makeop{Div}   \makeop{Lie} 
\makeop{PGL}   \makeop{Corr}  \makeop{PSL}   \makeop{sgn}   \makeop{Spf}
\makeop{Spec}  \makeop{Tr}    \makeop{Nr}    \makeop{Fr}    \makeop{disc}
\makeop{Proj}  \makeop{supp}  \makeop{ker}   \makeop{im}    \makeop{dom}
\makeop{coker} \makeop{Stab}  \makeop{SO}    \makeop{SL}    \makeop{SL}
\makeop{Cl}    \makeop{cond}  \makeop{Br}    \makeop{inv}   \makeop{rank}
\makeop{id}    \makeop{Fil}   \makeop{Frac}  \makeop{GL}    \makeop{SU}
\makeop{Trd}   \makeop{Sp}    \makeop{Tr}    \makeop{Trd}   \makeop{diag}
\makeop{Res}   \makeop{ind}   \makeop{depth} \makeop{Tr}    \makeop{st}
\makeop{Ad}    \makeop{Int}   \makeop{tr}    \makeop{Sym}   \makeop{can}
\makeop{length}\makeop{SO}    \makeop{torsion} \makeop{GSp} \makeop{Ker}
\def\makebb#1{\expandafter\def
  \csname bb#1\endcsname{{\mathbb{#1}}}\ignorespaces}
\def\makebf#1{\expandafter\def\csname bf#1\endcsname{{\bf
      #1}}\ignorespaces} 
\def\makegr#1{\expandafter\def
  \csname gr#1\endcsname{{\mathfrak{#1}}}\ignorespaces}
\def\makescr#1{\expandafter\def
  \csname scr#1\endcsname{{\EuScript{#1}}}\ignorespaces}
\def\makecal#1{\expandafter\def\csname cal#1\endcsname{{\mathcal
      #1}}\ignorespaces} 

\def\doLetters#1{#1A #1B #1C #1D #1E #1F #1G #1H #1I #1J #1K #1L #1M
                 #1N #1O #1P #1Q #1R #1S #1T #1U #1V #1W #1X #1Y #1Z}
\def\doletters#1{#1a #1b #1c #1d #1e #1f #1g #1h #1i #1j #1k #1l #1m
                 #1n #1o #1p #1q #1r #1s #1t #1u #1v #1w #1x #1y #1z}
\doLetters\makebb   \doLetters\makecal  \doLetters\makebf
\doLetters\makescr 
\doletters\makebf   \doLetters\makegr   \doletters\makegr
     \def\qed{\qedmark\medbreak}%
\def\qedmark{{\enspace\vrule height 6pt width 5pt depth 1.5pt}}%
    \def\setminus{\smallsetminus}

\normalsize

\makeop{Bl}

\def\Fpbar{\overline{\bbF}_p}
\def\Fp{{\bbF}_p}

\def\Qpbar{\overline{{\bbQ}_p}}
\def\Qp{{\bbQ}_p}

\def\Zp{{\bbZ}_p}
\def\Qbar{\overline{\bbQ}}

\newcommand{\Z}{\mathbb Z}

\newcommand{\C}{\mathbb C}
\newcommand{\A}{\mathbb A}    
\newcommand{\F}{\mathbb F}

\newcommand{\<}{\langle}   
\renewcommand{\>}{\rangle} 

\newcommand{\isoto}{\stackrel{\sim}{\to}}
\nc{\embed}{\hookrightarrow}

\newcommand{\dieu}{Dieudonn\'{e} }

\nc{\ol}{\overline}
\nc{\wt}{\widetilde}
\nc{\opp}{\mathrm{opp}}
\def\ul{\underline}

\makeop{Ram}
\makeop{Rep}

\newcommand\gfrac[2]{\genfrac{}{}{0pt}{}{#1}{#2}}
\newcommand\Mlin{M^{\rm lin}}
\newcommand\Msympl{M^{\rm sympl}}
\newcommand\Flag{\mathcal Flag}
\newcommand\isomarrow{\stackrel{\cong}{\longrightarrow}}
\newcommand\stdp{\Lambda}
\newcommand\stdt{\lambda}
\newcommand\Iw{\mathcal I}

\begin{document}
\renewcommand{\thefootnote}{\fnsymbol{footnote}}
\setcounter{footnote}{-1}
\numberwithin{equation}{section}

\title[Kottwitz-Rapoport strata and Deligne-Lusztig
varieties]{Supersingular Kottwitz-Rapoport strata\\and Deligne-Lusztig varieties}
\author{Ulrich G\"ortz}
\address[G\"ortz]{
Mathematisches Institut\\
Beringstr.~1\\
53115 Bonn\\
Germany}
\email{ugoertz@math.uni-bonn.de}
\thanks{G\"{o}rtz was partially supported by a Heisenberg grant and by the
SFB/TR 45 ``Periods, Moduli Spaces and Arithmetic of Algebraic Varieties''
of the DFG (German Research Foundation)}. 
\author{Chia-Fu Yu}
\address[Yu]{
Institute of Mathematics \\
Academia Sinica \\
128 Academia Rd.~Sec.~2, Nankang\\ 
Taipei, Taiwan \\ and NCTS (Taipei Office)}
\email{chiafu@math.sinica.edu.tw}
\thanks{Yu was partially supported by a NSC grant 
NSC 96-2115-M-001-001}.

\begin{abstract}
We investigate the special fibers of Siegel modular varieties with Iwahori
level structure. On these spaces, we have the Newton stratification, and
the Kottwitz-Rapoport stratification; one would like to understand how these
stratifications are related to each other. We give a simple description of
all KR strata which are entirely contained in the
supersingular locus as disjoint unions of Deligne-Lusztig varieties. We
also give an explicit numerical description of the KR stratification in
terms of abelian varieties.
\end{abstract}
 
\maketitle


\section{Introduction}
\label{sec:01}

Fix a prime number $p$ and an integer $g \ge 1$. The moduli space $\mathcal
A_g$ of principally polarized abelian varieties is an important variety
which has received a lot of attention over the last decades. In this paper
we are mainly concerned with a variant, the Siegel modular variety
$\mathcal A_{g,I}$ (which we usually abbreviate to $\mathcal A_I$) with
Iwahori level structure at $p$, which is much less well understood.  By
definition, $\mathcal A_I$ is the moduli space of the isomorphism
classes of chains of abelian varieties
\[
(A_0 \rightarrow A_1 \rightarrow \cdots \rightarrow A_g, \lambda_0,
\lambda_g, \eta),
\]
where the $A_i$ are abelian varieties of dimension $g$, the maps $A_i
\rightarrow A_{i+1}$ are isogenies of degree $p$, $\lambda_0$ and
$\lambda_g$ are principal polarizations of $A_0$ and $A_g$, respectively,
such that the pull-back of $\lambda_g$ is $p\lambda_0$, and $\eta$ is a
level structure away from $p$. See Section \ref{sec:02} for the precise
definition. We consider these spaces exclusively in positive
characteristic, i.~e.~over $\mathbb F_p$ or an algebraic closure $k \supset
\mathbb F_p$. The same definition makes sense over the ring $\mathbb Z_p$
of $p$-adic integers, and in particular in characteristic zero. In fact,
the motivation to study these spaces in positive characteristic is to
obtain arithmetic properties of the corresponding spaces over $\mathbb Q$.

We consider the following two stratifications of the space $\mathcal A_I$: The
Newton stratification is given by the isogeny type of the underlying
$p$-divisible groups of the abelian varieties in a chain as above. We are
particularly interested in the supersingular locus, i.~e.~the closed subset
of points $(A_i, \lambda_0, \lambda_g, \eta)$, where all the $A_i$ are
supersingular. Although for a general Newton stratum there is little hope
to achieve an explicit geometric description, in the case of the
supersingular locus
one can be more optimistic. Note that here we use the term
``stratification'' in the very loose sense that $\mathcal A_I$ is the
disjoint union of locally closed subsets; it is not true in general that
the closure of a stratum is a union of strata.

Similarly, one has the Newton stratification of the space $\mathcal A_g$.
In this case, the closure of each stratum is a union of strata. There are a
number of results describing the geometry of the supersingular locus in
$\mathcal A_g$. For instance, it was proved by Li and Oort \cite{li-oort}
that the dimension of the supersingular locus is $[g^2/4]$. There is also a
formula for the number of irreducible components, in terms of a certain
class number. As further references, besides \cite{li-oort} and the
references given there, we mention the articles \cite{koblitz:thesis} by
Koblitz, and \cite{yu:ss_siegel} by the second author.

On the other hand, in the Iwahori case, which is the case considered here,
currently very little is known. Even the dimension of the supersingular
locus is known only for $g\le 3$ (but our results in this paper and in
\cite{goertz-yu:kreo} prove that for even $g$ it is $g^2/2$). Note that
the situation here is definitely more complicated than in the case of good
reduction; as an example, in the case $g=2$, the supersingular locus
coincides with the $p$-rank $0$ locus, but it is not contained in the
closure of the $p$-rank $1$ locus. In addition, it is not
equi-dimensional (see \cite[Prop.~6.3]{yu:prank}).
Nevertheless, a better understanding of the geometric structure of the
supersingular locus seems within reach, and is clearly an interesting goal.

The second stratification is the Kottwitz-Rapoport stratification (KR
stratification)
\[
\mathcal A_I = \coprod_{x\in {\rm Adm}_I(\mu)} \mathcal A_{I,x}
\]
by locally closed subsets, which should be thought of as a stratification
by singularities (see Remark \ref{strat_by_sing}). It corresponds to the
stratification by Schubert cells of the associated local model. In terms of
abelian varieties, we can express this as follows: the strata are the loci
where the relative position of the chain of de Rham cohomology groups
$H^1_{DR}(A_i)$ and the chain of Hodge filtrations
$\omega(A_i) \subset H^1_{DR}(A_i)$ is constant. The KR stratification on the space $\mathcal A_g$
consists of only one stratum, and hence does not provide any interesting
information. See Section \ref{sec:02} for a reminder on the definition of
the KR stratification and on the set of strata, the so-called
$\mu$-admissible set ${\rm Adm}_I(\mu)$.  We give the following explicit
characterization of KR strata (see Corollary \ref{cor_num_char}):

\begin{thm}
Let $(A_i)_i$, $(A'_i)_i$ be $k$-valued points of $\mathcal A_I$.
Denote by $\alpha_{ji}$ the natural map $H^1_{DR}(A_i) \rightarrow
H^1_{DR}(A_j)$ and by $^\perp$ the orthogonal complement inside $H^1(A_0)$
with respect to the pairing induced by the principal polarization on $A_0$,
and similarly for the chain $(A'_i)_i$.

Then the points $(A_i)_i$, $(A'_i)_i$ lie in the same KR stratum if
and only if 
for all $0 \le j< i\le g$, one has
\begin{eqnarray*}
&& \dim \omega(A_j) / \alpha_{ji}(\omega(A_i)) =
 \dim \omega(A'_j) / \alpha'_{ji}(\omega(A_i)),
   \\
&& \dim H^1_{DR}(A_j)/(\omega(A_j) + \alpha_{ji}(H^1_{DR}(A_i))) =
 \dim H^1_{DR}(A'_j)/(\omega(A'_j) + \alpha'_{ji}(H^1_{DR}(A'_i))),
\end{eqnarray*}
and for all all $0 \le i, j\le g$,
\[ \dim \alpha_{0i}(\omega(A_i)) + \alpha_{0j}(H^1_{DR}(A_j))^\perp =
 \dim \alpha'_{0i}(\omega(A'_i)) +
 \alpha'_{0j}(H^1_{DR}(A'_j))^\perp. \]
\end{thm}

There is an explicit formula for these values on the stratum associated
with a given element of the $\mu$-admissible set; see Section~\ref{sec:02}.

The relationship between the Newton stratification and the KR
stratification is complicated.
In general neither of these stratifications is a
refinement of the other one. Nevertheless, there are some relations between
them. For instance, the ordinary Newton stratum (which is open and dense in
$\mathcal A_I$) is precisely the union of the maximal KR strata.  At the
other extreme, the supersingular locus is not in general a union of KR
strata. However, it is our impression that those KR strata which are
entirely contained in the supersingular locus make up a significant part of
it. We call these KR strata \emph{supersingular}. In Section
\ref{sec:superspecial_strata} we specify the subset of \emph{superspecial}
strata inside the set of all KR strata. These strata are supersingular by
definition, and in \cite{goertz-yu:kreo} we prove that the set of
superspecial KR strata coincides with the set of supersingular KR strata
(see Theorem \ref{thm_ss_KR_strata}). In Section
\ref{sec:structure_ssp_strata} we give a very simple geometric description
of the superspecial KR strata in
terms of Deligne-Lusztig varieties. 

According to our definition (Definition~\ref{43}), we define first
what an {\it $i$-superspcial} KR stratum, for $0\le i\le [\frac g2]$, is, 
and then call a KR stratum {\it superspecial} if it is 
$i$-superspecial for some integer $0\le i\le [\frac g2]$.
To give an idea, let us consider the case $i=0$ to simplify the notation.
We call a KR stratum \emph{$0$-superspecial}, if
for one (equivalently: all) chain $(A_i)_i$ lying in this stratum, $A_0$
and $A_g$ are superspecial, and the isogeny $A_0\rightarrow A_g$ given by
the chain is isomorphic to the Frobenius morphism $A_0\rightarrow
A_0^{(p)}$. Fix such a chain of abelian varieties, and denote by $G'$ the
automorphism group scheme of the principally polarized abelian variety
$(A_0,\lambda_0)$, i.~e.~$G'(R) = \{\,x\in (\End(A_0)\otimes R)^\times\ ;\
x' x=1\, \}$, where $'$ is the Rosati involution for $\lambda_0$. We
consider the base change of $G'$ over $\mathbb Q_p$; it is an inner form
of the derived group of $G=GSp_{2g}$. Let $\overline{G}'$ be the  
quasi-split unitary group which arises as the maximal
reductive quotient of the special fiber of the Bruhat-Tits group scheme for
the maximal parahoric subgroup of $G'$ which is the stabilizer of the
``vertex'' $\{0, g\}$ of the base alcove.
As indicated above, we parametrize the KR strata by the admissible set
${\rm Adm}_I(\mu)$, a finite subset of the extended affine Weyl
group. There is a unique 
element $\tau$ of length $0$ such that ${\rm Adm}_I(\mu) \subset W_a\tau$,
where $W_a$ is the affine Weyl group, a Coxeter group generated by simple
reflections $s_0, \dots, s_g$. In particular, $\mathcal A_\tau$ is the
unique $0$-dimensional KR stratum; it is contained in the closure of every
KR stratum. We have that $w \tau \in {\rm Adm}(\mu)$ gives rise to a
$0$-superspecial stratum if and only if $w$ lies in $W_{\{0,g\}}$, 
where $W_{\{0,g\}}$ is the subgroup of $W_a$ generated by $s_1,
\dots, s_{g-1}$. The group $W_{\{0,g\}}$ is isomorphic to the
symmetric group $S_g$ on $g$ letters; we see it as the Weyl group of
the group $\overline{G}'$.
We have, by Proposition
\ref{maxl_ss_kr}, Theorem \ref{sskr_are_dl} and Corollary
\ref{desc_ss_kr_strata}: 

\begin{thm}
Let $w\in W_{\{0,g\}}$, such that $\mathcal A_{w\tau}$ is a
$0$-superspecial KR stratum. We have an isomorphism
\[
\mathcal A_{w\tau}  \isomarrow \coprod_{x \in \pi(\mathcal A_\tau)} X(w^{-1}),
\]
where $X(w^{-1})$ is the Deligne-Lusztig variety associated to $w^{-1}$ in
the flag variety of all Borel subgroups of $\overline{G}'$,
and $\pi$ is the projection $\mathcal A_I \rightarrow \mathcal A_g$, which
maps a chain $(A_i,\lambda_0,\lambda_g,\eta)$ to $(A_0,\lambda_0,\eta)$.
\end{thm}

We also determine the number of connected components of $\mathcal
A_{w\tau}$ (with $w$ as above); see Corollary \ref{num_conn_comp}.
As indicated above, both results carry over to the case of 
$i$-superspecial strata for any $i\in\{0,\dots, [g/2]\}$.

Note that the union of supersingular KR strata is an interesting subvariety
of $\mathcal A_I$, not only because it has such a nice description, and in
fact a description which links it to representation theory. From the point
of view of the trace formula, it makes sense to restrict to a set of KR
strata (which corresponds to the choice of a particular test function) and,
at the same time, to a certain Newton stratum (the latter corresponds to
making the index set of the sum smaller).
It also becomes apparent through the results of \cite{goertz-yu:kreo} that
among the set of all KR strata, the superspecial ones are singled out in
several ways. For instance, all KR strata which are not superspecial, are
connected. 

Maybe most important, although the union of the
supersingular KR strata is not all of the supersingular locus, we still get
a significant part. The following table backs this up for small $g$.
First, we have the following result on the dimension (Proposition
\ref{dim_sspKR}):

\begin{prop}
The dimension of the union of all superspecial KR strata is $g^2/2$, if $g$
is even, and $g(g-1)/2$, if $g$ is odd. There is a unique superspecial
stratum of this maximal dimension.
\end{prop}

The
dimension of the whole moduli space $\mathcal A_I$ is $g(g+1)/2$. The
dimension of the union of all superspecial KR strata is given in
Proposition \ref{dim_sspKR}; it is $g^2/2$ if $g$ is even, and $g(g-1)/2$
otherwise. The numbers of KR strata, and of KR strata of $p$-rank $0$ can
be obtained from Haines' paper \cite{haines:bernstein}, Prop.~8.2, together
with the results of Ng\^o and Genestier \cite{ngo-genestier:alcoves}.
The dimension of the $p$-rank $0$ locus is $[g^2/2]$
(see \cite{goertz-yu:kreo}, Thm.~8.8).
It follows in particular that for $g$ even, the dimension of the
supersingular locus is $g^2/2$.
Note that for $g=5$ we do not know the dimension of the
supersingular locus; for $g=6$ we know it only because it has to lie
between the dimension of the union of all superspecial KR strata and the
dimension of the $p$-rank $0$ locus. As a word of warning one should say
that neither of these loci is equi-dimensional in general.

\[
\begin{array}{|l|c|c|c|c|c|c|}\hline
g & 1 & 2 & 3 & 4 & 5 & 6 \\\hline
\text{number of KR strata} & 3 & 13 & 79 & 633 & 6331 & 75973 \\\hline
\text{number of KR strata of $p$-rank } 0 & 1 & 5 & 29 & 233 & 2329 &
27949 \\\hline 
\text{dim.~of union of superspecial KR strata} & 0 & 2 & 3 & 8 & 10 &
18 \\\hline 
\text{dim.~of supersingular locus} & 0 & 2 & 3 & 8 & ? & 18 \\\hline
\text{dim.~of $p$-rank $0$ locus} & 0 & 2 & 4 & 8 & 12 & 18 \\\hline
\dim \mathcal A_I & 1 & 3 & 6 & 10 & 15 & 21 \\\hline
\end{array}
\]

Furthermore, it can be shown that any irreducible component of maximal
dimension of the union of all superspecial KR strata is actually an
irreducible component of the $p$-rank $0$ locus, and hence in particular an
irreducible component of the supersingular locus. Also see the remarks
at the end of Section~\ref{sec:superspecial_strata} and in particular
\cite{goertz-yu:kreo}.

Deligne-Lusztig varieties play a prominent role in the representation
theory of finite groups of Lie type. More precisely, one can realize the
representations of these groups in their cohomology (with coefficients in
certain local systems). On the other hand, generally speaking, it is
suggested by the work of Boyer \cite{boyer:llc}, Fargues
\cite{fargues:thesis}, Harris and Taylor
\cite{harris:ecm2000,harris:ihp2000,harris-taylor:llc}, and others that a
correspondence of Jacquet-Langlands type should be realized in the
cohomology of the supersingular locus. Of course, one will have to consider
deeper level structure. Nevertheless, since the supersingular locus (or
even the union of superspecial KR strata) is of quite high dimension, and
with the link to Deligne-Lusztig varieties we have an interesting
connection to representation theory, it is of high interest to
investigate which representations occur in the cohomology of the
superspecial KR strata. We remark that the local component $\Pi_p$ at $p$ of an 
admissible representation $\Pi_f$ of $GSp_{2g}(\A_f)$ which occurs in the 
cohomology of the moduli space $\calA_I$ has a non-zero Iwahori fixed vector.
It is proved by Borel \cite{borel:Iwahori} that (a) 
any subquotient of an unramified principal series contains a non-zero
  Iwahori fixed vector, and (b)
any irreducible admissible representation which possesses a non-zero
  Iwahori fixed vector occurs as a subquotient of an unramified principal
  series.

There is a third stratification on the moduli space $\mathcal A_g$, the
so-called Ekedahl-Oort stratification. It
is given by the isomorphism class of the $p$-torsion (as a finite group
scheme) of the underlying abelian variety. See Oort's paper \cite{oort01}.
The relationship between the Newton stratification and the
EO stratification has been studied by Harashita; see for instance
\cite{harashita} where Deligne-Lusztig varieties also make an appearance.
About the relationship between the KR stratification on $\mathcal A_I$ and
the EO stratification on $\mathcal A_g$, not much is known at
present. 
It is easy to show that the image of every KR stratum under the natural map
is a union of certain EO strata.  In particular, all supersingular KR
strata are contained in the inverse image of the union of all EO strata
which are contained in the supersingular locus of $\mathcal A_g$. 
It is not clear to us whether this inclusion is an equality.  
We also mention the paper \cite{ekedahl-vdgeer} by Ekedahl and 
van der Geer, in which the EO
stratification of $\mathcal A_g$ is investigated using a \emph{flag
complex} $\mathcal F_g \rightarrow \mathcal A_g$, which contains part of
$\mathcal A_I$. See \cite{goertz-yu:kreo} for some results about the
relationship between the KR stratification and the Ekedahl-Oort
stratification.

Deligne-Lusztig varieties also play a role in recent work
of Yoshida \cite{yoshida} who gives a local approach to non-abelian
Lubin-Tate theory (in the case of depth $0$), and
of Vollaard and Wedhorn \cite{vollaard}, \cite{vollaard-wedhorn} about the
supersingular locus of the Shimura variety of $GU(1,n-1)$ in the case of
good reduction. There is also the paper \cite{hoeve} by Hoeve which
appeared very shortly after the first version of the current article,
where a description of Ekedahl-Oort strata which are entirely contained in
the supersingular locus is given in terms of Deligne-Lusztig varieties.
This refines the results of Harashita \cite{harashita}. 
It seems that there must be a relationship to our description of 
superspecial KR strata; it would be interesting to understand it 
precisely.

We conclude the introduction with an overview about the individual
sections. In Section 2, we recall the group theoretic notation which we
use, as well as the definition of local models and of the KR
stratification. We give an explicit numerical characterization of KR strata
in terms of abelian varieties (Cor.~\ref{cor_num_char}). In Section 3 we
assemble some results about the minimal KR stratum. In Section 4 we
construct the list of superspecial KR strata, prove that they are
supersingular, and compute the dimension of their union,
and in Section 5 we recall the definition and some basic properties of
Deligne-Lusztig varieties. The main result about the description of the
supersingular KR strata constructed in Section 4 as disjoint unions of
Deligne-Lusztig varieties is given in Section 6.  We close with a few
remarks about generalizations to other Shimura varieties, specifically to
the unitary case in Section 7.

\section{Local models and numerical characterization of strata}
\label{sec:02}

In Sections \ref{subsec:2.1}--\ref{sec:24} we recall some notation and
collect a number of previously known facts about the moduli spaces we are
concerned with, and the KR stratification.

\subsection{The extended affine Weyl group}
\label{subsec:2.1}

We fix a complete discrete valuation ring $\calO$ with fraction
field $K$, uniformizer $\pi\in\calO$, and residue class field
$\kappa:=\calO/\pi\calO$.

The groups which are of primary interest for us in the sequel are $G=GL_n$
and $G=GSp_{2g}$. In these cases, everything is easily made explicit (and
we will mostly do so, below). It turns out that (in contrast to the general
case) ``all'' notions are well behaved with respect to the natural
embedding $GSp_{2g} \subset GL_{2g}$.

Let $G$ be a split connected reductive group over $\calO$, let $T$ be
a split maximal torus, 
and let $B\supseteq T$ be a Borel subgroup of $G$.  These data give rise to
a based root datum $(X^*, X_*, R, R^\vee, \Delta)$.  We will assume for
simplicity that it is irreducible. Denote by $W = N_GT/T$ its Weyl group,
generated by the simple reflections $\{ s_\alpha;\ \alpha \in \Delta \}$.
Denote by $\widetilde{W} := X_* \rtimes W \cong
N_GT(\kappa((t)))/T(\kappa[[t]])$ the extended affine Weyl group.
For $\lambda\in X_* = X_*(T)$, we denote by $t^\lambda$ the corresponding
element in $\widetilde{W}$. We also use the notation $t^\lambda$ for the
element $\lambda(t)$ of $G(\kappa((t)))$ (where we regard $\lambda$ as a
homomorphism $\mathbb G_m\rightarrow G$).
We may regard the group $\wt W$ as a subgroup of the group
$\bfA(X_{*,\mathbb R})$ of affine transformations on the space
$X_{*,\mathbb R}$.
The element $x=t^\nu w$ is identified with the
function $x(v)=w\cdot v+\nu$ for $v\in X_{*,\mathbb R}$.
Let $S_a = \{ s_\alpha;\ \alpha
\in \Delta\} \cup \{ s_0 \}$, where $s_0 =
t^{-\tilde{\alpha}^\vee}s_{\tilde{\alpha}}$ and where
$\tilde{\alpha}$ is the unique highest root. The subgroup $W_a \subseteq
\widetilde{W}$ generated by $S_a$ is the affine Weyl group of the root
system associated with our root datum, and $(W_a, S_a)$ is a Coxeter
system.

We define a length function $\ell \colon \widetilde{W} \longrightarrow
\mathbb Z$ as follows:
\begin{equation}
  \label{eq:21}
  \ell(w t^\lambda) = \sum_{\gfrac{\alpha<0}{w(\alpha) > 0}} |
\langle \alpha, \lambda \rangle +1 | + \sum_{\gfrac{\alpha<0}{w(\alpha)<0}}
| \langle \alpha,\lambda \rangle |
\end{equation}
This function extends the length function on $W_a$.
We have a short exact sequence
\begin{equation}
  \label{eq:22}
  1 \longrightarrow  W_a \longrightarrow \widetilde{W} \longrightarrow
  X_*/Q^\vee \longrightarrow 0,
\end{equation}
where $Q^\vee$ is the coroot lattice, i.~e.~the subgroup of $X_*$ generated
by $R^\vee$. The restriction of the projection $\widetilde{W}
\longrightarrow X_*/Q^\vee$ to the subgroup $\Omega \subseteq
\widetilde{W}$ of elements of length $0$ is an isomorphism $\Omega
\isomarrow X_*/Q^\vee$. The group $X_*/Q^\vee$ is called the algebraic
fundamental group of $G$, and is sometimes denoted by $\pi_1(G)$.

We extend the Bruhat order on $W_a$ to $\wt W$ by declaring 
\begin{equation}
  \label{eq:23}
  w\tau \le w'\tau' \Longleftrightarrow w \le w', \tau=\tau',\qquad
w,w'\in W_a,\ \tau, \tau' \in \Omega.
\end{equation}

For an affine root $\beta = \alpha - n$, $\alpha\in R$, $n \in \mathbb Z$,
we have the hyperplane $H_\beta = H_{\alpha,n} = \{ x \in X_{*,\mathbb R};\
\langle \alpha, x \rangle = n \}$ in $X_{*,\mathbb R} := X_*
\otimes_{\mathbb Z} \mathbb R$.  An alcove is a connected component of
the complement of the union of all affine root hyperplanes inside
$X_{*,\mathbb R}$. There is a unique alcove lying in the
\emph{anti-dominant} chamber (with respect to $\Delta$) whose closure
contains the 
origin. We call this alcove the base alcove and denote it by $\mathbf a$.
The group $\widetilde{W}$ acts on
$X_{*,\mathbb R}$, and since the union of all affine root hyperplanes is
stable under this action, we have an action of $\widetilde{W}$ on the set
of alcoves. The affine Weyl group $W_a$ acts simply transitively on the set
of alcoves, so we can identify $W_a$ with the set of alcoves in the
standard apartment $X_{*,\mathbb R}$ by mapping $w \in W_a$ to the alcove
$w\mathbf a$. On the other hand, the group $\Omega$ of elements of length
$0$ in $\widetilde{W}$ is precisely the stabilizer of the base alcove
inside $\widetilde{W}$. If $\lambda$ denotes the image of the origin under
$\tau \in \Omega$, then $\tau w_0 w = t^\lambda$, where $w$ is
the longest element in the stabilizer $W_\lambda$ of $\lambda$ in $W$, and
$w_0$ is the longest element in $W$.

\subsection{The general linear group}
\label{gln}

If $G = GL_n$, we choose $T$ to be the diagonal torus, and let $B$ be the
Borel subgroup of upper triangular matrices. We identify $X_*(T)$ with
$\mathbb Z^n$, and the Weyl group with the symmetric group $S_n$.
The fundamental group $\pi_1(GL_n)$ is isomorphic to $\mathbb Z$.
The standard lattice chain over $\calO$ is the chain
\begin{equation}
  \label{eq:24}
 \stdp_0 =\calO^{n} \subset \stdp_1 \subset \dots \subset
 \stdp_{n-1}\subset \stdp_{n} = \pi^{-1} \stdp_0  
\end{equation}
of $\calO$-lattices in $V$ where the lattice $\stdp_{i}$ is
defined as
\[
\stdp_{i} = \langle \pi^{-1}e_1,\dots, \pi^{-1}e_{i},
e_{i+1},\dots,e_{n}\rangle.
\]
(The
$\calO$-submodule in $V$ generated by $x_1,\dots, x_k\in V$ is denoted
by $\langle x_1,\dots, x_k\rangle$.) We extend the lattice chain to
$\stdp_i$ with 
index set $i\in \Z$ by putting $\stdp_{i+n}=\pi^{-1} \stdp_i$.
The stabilizer of the standard lattice chain is the Iwahori subgroup
$\Iw$ associated 
with $B$ (i.~e.~$\Iw$ is the inverse image of $B(\kappa)$ under the
projection $G(\calO) \rightarrow G(\kappa)$). We call $\Iw$ the standard
Iwahori subgroup. The base alcove $\mathbf a_1$ corresponding to $\Iw$
is the alcove whose stabilizer is $\Iw$.
Similarly, we have the standard lattice chain $(\stdt_i)_i$ over
$\kappa[[t]]$.

We denote by $\Flag_{GL_n}$ the affine flag variety for $GL_n$ over
$\kappa$. It 
parametrizes complete periodic lattice chains in $\kappa((t))^n$. See
\cite{beauville-laszlo}, \cite{faltings:loop-groups} for
details. Using the 
chain $(\stdt_i)_i$ as the base point, we can identify $\Flag_{GL_n}$
with the quotient (as fppf sheaves) of the loop group $G(\kappa((t)))$
by the stabilizer of 
$(\stdt_i)_i$, the standard Iwahori subgroup of $G(\kappa((t)))$.

Let us recall the definition of the local model attached to the general
linear group; see \cite{rapoport-zink}. We refer to this case as the \emph{linear case}. Fix a 
positive integer $r$, $0 < r < n$, and let $\Mlin$ be the $\calO$-scheme
representing the following functor. For an $\calO$-algebra $R$, let
\begin{equation}
  \label{eq:25}
\Mlin(R) = \{ (\scrF_i)_i \in \prod_{i=0}^{n-1} \mathop{\rm
Grass}\nolimits_{r}(\stdp_i)(R);\ \alpha_i(\scrF_i)
\subseteq \scrF_{i+1} \text{ for all } i=0, \dots, n-1\},
\end{equation}
where $\alpha_i \colon \stdp_i \otimes R \rightarrow \stdp_{i+1}
\otimes R$ is the map obtained by base change from the inclusion $\stdp_i
\subset \stdp_{i+1}$.  In terms of the natural bases, it is given by
the matrix ${\rm diag}(1, \dots, 1, \pi, 1,\dots, 1)$ 
with the $\pi$ in the $(i+1)$-th place. Note that we impose the
condition for $i=n-1$, too, where we set $\scrF_n := \scrF_0$.
We can similarly define a variant $\Mlin_J$, if instead of $I$ we use a
subset $J \subset I$ as the index set, i.~e.~where we consider compatible
families $(\scrF_i)_{i\in J}$.

We can
identify the special fiber $\Mlin_\kappa$ of the local model with the
closed subscheme
\begin{equation}
  \label{eq:26}
\{ (\scrL_i)_i;\ t \stdt_i \subseteq \scrL_i \subseteq \stdt_i,\
\bigwedge^n \scrL_0 = t^{n-r} \kappa[[t]]\}  
\end{equation}
of the affine flag variety (it is clear how to understand the above
description in a functorial way, and that this functor coincides with the
functor represented by the local model; see also \cite{goertz:symplectic}).
See \cite{rapoport-zink}, ch.~3 for a more general definition of local
models and for the relationship to Shimura varieties and moduli spaces of
$p$-divisible groups. From the point of view of Shimura varieties, the
local model defined above occurs in the case of unitary groups which split
over an unramified extension of $\mathbb Q_p$.

\subsection{The group of symplectic similitudes}

Set $V:=K^{2g}$ and let $e_1,\dots,e_{2g}$ be the standard basis. Denote by
$\psi:V\times V\to K$ the non-degenerate alternating form whose non-zero
pairings are 
\begin{eqnarray*}
&& \psi(e_i,e_{2g+1-i})=1, \quad 1\le i\le g,\\
&& \psi(e_i,e_{2g+1-i})=-1, \quad g+1 \le i\le 2g.
\end{eqnarray*}
The representing matrix for
$\psi$ is 
\[ 
\begin{pmatrix}
  0 & \wt I_g \\
  -\wt I_g & 0 
\end{pmatrix},\quad \wt I_g=\text{anti-diag(1,\dots,1)}.\]

Let $GSp_{2g}$ be the group of symplectic
similitudes with respect to $\psi$ (over $\calO$ or over $\kappa$ etc.,
depending on the context). 
Denote by $T$ the diagonal torus of $GSp_{2g}$, and by $B$ the Borel group
of upper triangular matrices.

The standard embedding $GSp_{2g} \subset GL_{2g}$ gives us an
identification of the Weyl group $W = N_GT/T$ with a subgroup of the Weyl
group of $GL_{2g}$. If we identify the latter as the symmetric group
$S_{2g}$ in the usual way, then $W$ is the subgroup 
consisting of elements that commute with the permutation
\begin{equation}
  \theta=(1,2g)(2,2g-1)\dots(g,g+1).
\end{equation}

Similarly, we identify the cocharacter group $X_*(T)$ of $T$ with the
group  
\begin{equation}
\{(u_1,\dots, u_{2g})\in \Z^{2g}\, |\, u_1+u_{2g}=\dots=u_{g}+u_{g+1} \}.  
\end{equation}

The extended affine Weyl group $\wt W$ is the semi-direct product
$X_*(T)\rtimes W$ of the finite Weyl group $W$ and the cocharacter group
$X_*(T)$.
The affine Weyl group is generated by
the simple affine reflections $s_0,\dots, s_g$ which we can express (with
respect to the identification 
 $\wt W \subset \wt W_{GL_{2g}} = \mathbb Z^{2g} \rtimes S_{2g}$) as
\begin{equation}\label{desc_si}
  \begin{array}{l}
  s_i=(i,i+1) (2g+1-i, 2g-i), \quad i=1,\dots,g-1,\\
  s_g=(g,g+1),\quad s_0=(-1,0,\dots,0,1), (1,2g).
  \end{array}
\end{equation}

The fundamental group of $GSp_{2g}$ is $\mathbb Z$. If $\tau$ denotes a
generator of the subgroup $\Omega\subset \wt W$ of elements of length
zero, 
then $\tau^2$ is central. In fact, $\tau^2$ is $t_{(1, \dots, 1)}$ or
$t_{(-1,\dots, -1)}$; later we will work with $\tau$ such that $\tau^2 =
t_{(1,\dots, 1)}$.

Let $\mu = (1,\dots,1,0,\dots,0)$ be the minuscule dominant coweight,
associated with the Shimura variety of Siegel type.

The standard lattice chains $(\stdp_i)_i$ (in $K^{2g}$) and
$(\stdt_i)_i$ (in $\kappa((t))^{2g}$) defined above are self-dual. We
define the standard Iwahori subgroup $\Iw$ (in $GSp_{2g}(K)$, and
$GSp_{2g}(\kappa((t)))$, resp.) to be the stabilizer of the standard lattice
chain. The standard lattice chain also gives rise to a base alcove in the
apartment $X_*(T)_{\mathbb R}$.

We have the affine flag variety $\Flag_{GSp_{2g}}$ for $GSp_{2g}$ which
again we can define as the quotient of the loop group $GSp_{2g}(k((t)))$ by
the Iwahori subgroup fixed above.
The inclusion $GSp_{2g} \subset GL_{2g}$ induces a 
closed embedding of $\Flag_{GSp_{2g}}$ into the affine flag variety for
$GL_{2g}$. In this way, we can identify $\Flag_{GSp_{2g}}$ with the locus of
all self-dual lattice chains in $\Flag_{GL_{2g}}$.

Now we want to give the definition of the local model for the symplectic
group. Let $r=g$, $n = 2g$. For this choice of $r$ and $n$
we get a linear local model $\Mlin$ as defined above, and we let $\Msympl$
be the $\calO$-scheme representing the functor
\[
\Msympl(R) = \{ (\scrF_i)_i \in \Mlin(R);\ 
\forall i: \scrF_i \rightarrow \stdp_{i}\otimes R \cong
(\stdp_{2g-i}\otimes R)^\vee \rightarrow \scrF_{2g-i}^\vee \text{ is
  the zero map} \} 
\]
where $R$ is any $\calO$-algebra, $\stdp_{i}\otimes R
\cong (\stdp_{2g-i}\otimes R)^\vee$ is the isomorphism induced by
$\psi$ (and the 
periodicity isomorphism $\stdp_{-i} \cong \stdp_{2g-i}$), and
$(\stdp_{2g-i}\otimes R)^\vee \rightarrow \scrF_{2g-i}^\vee$ is the
$R$-dual of the inclusion $\scrF_{2g-i} \rightarrow
\stdp_{2g-i}\otimes R$. Clearly, 
$\Msympl$ is represented by a closed subscheme of $\Mlin$.  Because of the
duality condition, it is enough to keep track of the partial chain
$\scrF_i$, $i=0,\dots, g$. Because of the periodicity, we can also use the
lattice chain $\stdp_{-g} \rightarrow \cdots \rightarrow \stdp_{0}$
instead of $\stdp_{0} \rightarrow \cdots \rightarrow \stdp_{g}$.

\subsection{Moduli spaces and the KR stratification}
\label{sec:24}
Let $g\ge 1$ be an integer, $p$ a rational prime, $N\ge 3$ an integer
with $(p,N)=1$. Choose $\zeta_N\in \Qbar \subset \C$ a primitive
$N$th root of unity and fix an embedding $\Qbar\hookrightarrow \Qpbar$. 
Put $I:=\{0,1,\dots,g\}$. Let $\calA_I$ be the moduli space over
$\Fpbar$ parametrizing equivalence classes of objects
\[ (A_0\stackrel{\alpha}{\to} A_1\stackrel{\alpha}{\to}\dots
\stackrel{\alpha}{\to} A_g, \lambda_0,\lambda_g,\eta), \]
where 
\begin{itemize}
\item each $A_i$ is a $g$-dimensional abelian variety,
\item $\alpha$ is an isogeny of degree $p$,
\item $\lambda_0$ and $\lambda_g$ are principal polarizations on $A_0$
  and $A_g$, respectively, such that $(\alpha^g)^* \lambda_g=p\lambda_0$.
\item $\eta$ is a symplectic level-$N$ structure on $A_0$
  w.r.t. $\zeta_N$. 
\end{itemize}
Put $\eta_0:=\eta$, $\eta_i:=\alpha_* \eta_{i-1}$ for $i=1,\dots, g$,
and $\lambda_{i-1}:=\alpha^* \lambda_i$ for $i=g,\dots, 2$. Let $\ul
A_i:=(A_i,\lambda_i,\eta_i)$. Then $\calA_I$ parametrizes equivalence
classes of  objects
\[ (\ul A_0\stackrel{\alpha}{\to} \ul A_1\stackrel{\alpha}{\to}\dots
\stackrel{\alpha}{\to} \ul A_g), \]
where $\ul A_0\in \calA_{g,1,N}$, and for $i\neq 0$, 
\[ \ul A_i\in \calA'_{g,p^{g-i},N}:=\{\ul A\in \calA_{g,p^{g-i},N}\,
|\, \ker \lambda\subset A[p]\, \}. \]
Here $\calA_{g,d,N}$ denotes the moduli space of abelian varieties of
dimension $g$ with a polarization of degree $d^2$ and a symplectic
level-$N$-structure.
For any non-empty subset $J=\{i_0,\dots, i_r\}\subset I$, let $\calA_J$
be the moduli space over $\Fpbar$ parametrizing equivalence classes of 
objects
\[  (\ul A_{i_0} \stackrel{\alpha}{\to} \ul
A_{i_1}\stackrel{\alpha}{\to}\dots 
\stackrel{\alpha}{\to} \ul A_{i_r}), \]
where $\ul A_{i_0}\in \calA_{g,1,N}$ if $i_0=0$, and $\ul A_{i_j}\in
\calA'_{g,p^{g-i_j},N}$ for others.

For $J_1\subset J_2$, let $\pi_{J_1,J_2}:\calA_{J_2}\to \calA_{J_1}$
be the natural projection. The transition morphism $\pi_{J_1,J_2}$ is
proper and dominant. For $1\le j\le r$, put $k_j:=i_j-i_{j-1}$ and put
$k_0:=0$. There are the following fundamental results

\begin{fact}\label{21}\
 \begin{itemize}
\item[(1)] The ordinary locus $\calA_J^{\rm ord}\subset \calA_J$ is dense.

\item[(2)] The moduli space $\calA_J$ has pure dimension $g(g+1)/2$.

\item [(3)] The moduli space $\calA_J$ has $(k_0+1)(k_1+1)\dots(k_r+1)$ irreducible
  components. 
\end{itemize} 
\end{fact}

\begin{proof}
Part (1) is proved in Ng\^o-Genestier
\cite{ngo-genestier:alcoves} for the case $J=I$ and in \cite{yu:gamma}
in general. Part (2) follows from (1) (or from the flatness of the
integral model of $\calA_J$, see \cite{goertz:symplectic}). 
Part (3)
is proved in \cite{yu:gamma} in the case $0\in J$. However, the general
case follows easily: consider
$J'=J\cup\{0\}$ and let $\calA_{J'}^{\rm ord, e}\subset
\calA_{J'}^{\rm ord}$ be the subvariety consisting of objects $(\ul
A_\bullet)$ such that $\ker(A_0\to A_{i_0})$ is etale. Then the
argument of loc.~cit.~shows that $\calA_{J'}^{\rm ord, e}$ has
$\prod_{i=0}^r (k_i+1)$ irreducible components, and that
$\calA_{J'}^{\rm ord, e}$ and $\calA^{\rm ord}_J$ have same number of
irreducible components. This proves (3) for all $J$.  
Note that this result for the case $|J|=1$ is also obtained in 
de Jong \cite{dejong:ag}.
\qed
\end{proof}

Now let $\mathcal O = \mathbb Z_p$, $V=\Qp^{2g}$, and let 
$\psi$, and the lattice chain $(\stdp_{i})_{i\in\mathbb Z}$ be as above.
Put $\psi_0:=\psi$ on $\stdp_0= \mathbb Z_p^{2g}$. 
Let $\psi_{-g}$ be the pairing on $\Lambda_{-g}$ which is
$\frac{1}{p}$ times the pull-back of $\psi_0$.
We shall also write $M^{\rm loc}_I$ for the local model $\Msympl$ 
associated to this lattice chain $\stdp_\bullet$ (here $I$ stands for the
Iwahori case, i.~e.~$I=\{ 0, \dots, g \}$). 

Let $\wt \calA_I$ be the moduli space over $\Fpbar$ parametrizing
equivalence classes of 
objects $(\ul A_\bullet,\xi)$, where $\ul A_\bullet\in \calA_I$ and
$\xi=(\xi_i)_{i\in I}:H^1_{\rm DR}(A_i/S)\simeq \stdp_{-i}\otimes
\calO_S$ is an isomorphism of lattice chains which
preserves the polarizations up to
scalars. Taking duals, the trivializations $\xi$ of the chain of 
de Rham cohomologies $H^1_{\rm DR}(A_i/S)$ are in one-to-one
correspondence with those of the chain of its dual 
$H_1^{\rm DR}(A_i/S)$ with the lattice chain $\{\stdp_i\otimes
\mathcal O_S\}_{i\in 
I}$. Here $H_1^{\rm DR}(A_i/S)$ is the linear dual of $H^1_{\rm
DR}(A_i/S)$. Let $\calG_I$ be the group scheme over $\Zp$
representing 
the functor $S\mapsto \Aut (\stdp_I\otimes
\calO_S,[\psi_0],[\psi_{-g}])$. The group scheme $\calG_I$ is smooth
and affine. This group acts on $\wt \calA_I$ and 
$\bfM^{\rm loc}_{I}$ from the left in the obvious way.

Following Rapoport and Zink \cite{rapoport-zink}, we have the
following local model diagram: 
\begin{equation}\label{eq:29}
  \xymatrix{
 & \wt \calA_I \ar[ld]_{\varphi_1} \ar[rd]^{\varphi_2} & \\
\calA_I & & \bfM^{\rm loc}_{I,\Fpbar},   
}
\end{equation}
where
\begin{itemize}
\item The morphism $\varphi_2$ is given by sending each object
  $(\ul A_\bullet,\xi)$ to the image $\xi(\omega_\bullet)$ of the
  Hodge filtration $\omega_\bullet \subset H^1_{\rm DR} (A_\bullet)$.
  This map is $\calG_I$-equivariant, surjective and smooth. 

\item The morphism $\varphi_1:\wt \calA_I\to \calA_I$ is simply
  forgetting the trivialization $\xi$; 
  this map is a $\calG_I$-torsor, and hence is smooth and
  affine.
\end{itemize}

Similarly we define the local model $\bfM^{\rm loc}_{J}$, the
moduli space $\wt \calA_J$, and the group scheme $\calG_J$ for each
non-empty subset $J\subset I$. We also have the
local model diagram between $\calA_J$, $\wt \calA_J$ and 
$\bfM^{\rm loc}_{J,\Fpbar}$ and the properties as above.
Although in this paper we are only concerned with these spaces in positive
characteristic, we remark that this construction
can be carried through over $\mathbb Z_p$ instead of $\mathbb F_p$.

Consider the decomposition into $\calG_I$-orbits, and its pullback to
$\wt \calA_I$:
\begin{equation}
  \bfM^{\rm loc}_{I,\Fpbar}=\coprod_{x} \bfM^{\rm loc}_{I,x},\quad 
\wt \calA_I=\coprod_{x}\wt \calA_{I,x}.
\end{equation}
Since $\varphi$ is a $\calG_I$-torsor, the stratification on 
$\wt \calA_I$ descends to a stratification, 
\begin{equation}
 \calA_I=\coprod_{x\in {\rm Adm}_I(\mu)} \calA_{I,x}. 
\end{equation}
This is called the Kottwitz-Rapoport (KR) stratification. 
We sometimes write $\calA_w$ instead of $\calA_{I,w}$.
The strata are indexed by the (finite) set ${\rm Adm}_I(\mu)$ of
$\mu$-admissible elements in the extended Weyl group $\wt W$. 
We recall the definition: 
\begin{equation}
{\rm Adm}_I(\mu)=\{x\in \wt W\, ; x\le t_{w(\mu)} \text{ for some
  $w\in W$}\, \}.\\
\end{equation}
Kottwitz and Rapoport \cite{kottwitz-rapoport:alcoves} have shown that
${\rm Adm}_I(\mu)$ is precisely the set of $\mu$-permissible alcoves as
defined in the following section (Def.~\ref{muperm}). 

In fact, the set ${\rm Adm}_I(\mu)$ is contained in $W_a \tau$, where 
 $\tau$ is the unique element that is less than $t^\mu$ and fixes 
 the base alcove $\mathbf a$.
In terms of the identification of $\wt W$ with a subgroup of $\wt
W_{GL_{2g}} = \mathbb Z^{2g} \rtimes S_{2g}$, we have 
\begin{equation}
  \tau=t_{(0,\dots,0,1,\dots, 1)} (1,g+1)(2,g+2)\dots(g,2g). 
\end{equation} 
We also note the following results
\begin{fact}\
  
{\rm (1)} Each KR stratum $\calA_{I,x}$ is smooth of pure dimension
$\ell(x)$.  

{\rm (2)} {\rm (Ng\^o-Genestier \cite{ngo-genestier:alcoves})} The
  $p$-rank function is constant on each KR 
  stratum. Furthermore, one has 
\[ p{\rm -rank}(x)=\frac{1}{2} \# {\rm Fix}(w), \]
where we write $x=t^\nu w$ and ${\rm Fix}(w):=\{i\in\{1,\dots, 2g\};
w(i)=i\}$ (and consider $w\in W\subset S_{2g}$).
\end{fact}

\begin{remark} \label{strat_by_sing}
We suggest to think of the KR stratification as a stratification by
singularities. One justification is the following: Let
$R\Psi\overline{\mathbb Q}_\ell$ be the sheaf of nearby cycles (where we
apply the nearby cycles functor to the constant sheaf $\overline{\mathbb
Q}_\ell$, $\ell\ne p$ a prime,  on the generic fiber over $\mathbb Q_p$ to
obtain a sheaf on the special fiber over $k$). Then the trace of Frobenius
on the stalks of $R\Psi\overline{\mathbb Q}_\ell$ is constant along the KR
strata. The reason is that the stalk at a point in the special fiber, and
the stalk at a point of the special fiber of the local model corresponding
to it via the local model diagram are isomorphic. On the local model,
however, the stratification is the stratification by orbits of a group
action, and the nearby cycles sheaf is equivariant for this action, so
clearly the stalk, and in particular the trace of Frobenius do not depend
on the choice of point in a fixed orbit.
Cf.~the paper \cite{haines-ngo:nearby-cycles} by Haines and Ng\^o,
Section 4.1.  

However, clearly points of different strata can have the same singularity,
i.~e.~can have smoothly equivalent, or even isomorphic stalks. So the
notion ``stratification by singularities'' has to be taken with a grain of
salt.
\end{remark}

\subsection{Numerical characterization of Schubert cells in the linear case}
\label{sec:23}
We recall the combinatorial description of ``alcoves''
of \cite{kottwitz-rapoport:alcoves} 
(which we, however, have to adapt to our normalization).
We write $\mathbf 1 = (1,\dots, 1)\in \mathbb Z^n$.
An extended alcove is a tuple $(x_i)_{i=0,\dots, n-1}$, $x_i \in
\mathbb Z^n$ such 
that, setting $x_n := x_0 - \mathbf 1$, for all $i\in \{0,\dots, n-1\}$ we
have $x_{i+1}(j) = x_{i}(j)$ for all but precisely one $j$, 
where $x_{i+1}(j) = x_{i}(j) - 1$.

The relationship between alcoves and extended alcoves is as follows.
We identify the fixed diagonal torus $T \subset GL_n$ with $\mathbb G_m^n$
in the obvious way, and correspondingly have $X_*(T) = \mathbb Z^n$. Hence
the entries $x_i$ of an extended alcove $(x_i)_i$ can be viewed as elements
of $X_*(T)$, and the conditions we impose ensure that they are the vertices
of an alcove in $X_*(T)_{\mathbb R}$. On the other hand, we can associate
to each extended alcove the sum $\sum_j x_0(j) \in \mathbb Z =
\pi_1(GL_n)$. 
The choices we made yield an identification  $\widetilde{W} = S_n \ltimes
\mathbb Z^n$, and we see that $\widetilde{W}$ acts simply transitively on the set of extended alcoves.
The alcove $(\omega_i)_i$, where $\omega_i =
(-1^{(i)}, 0^{(n-i)})$, is called the standard alcove. We use it as a base
point to identify $\widetilde{W}$ with the set of extended alcoves.
Similarly, we identify the affine Weyl group $W_a$ with the set of alcoves
in $X_*(T)_{\mathbb R}$ (as defined above). We obtain a commutative diagram
(the short exact sequence in the first row was discussed above)
\[
\xymatrix{
    W_a \ar[r] \ar[d]^{\cong} & \widetilde{W} \ar[d]^{\cong} \ar[r] &
    \pi_1(GL_n) \ar[d]^{=} \\
    \{ \text{alc.~in } X_*(T)_{\mathbb R} \}\ar[r] & \{ \text{ext'd alc.}
    \} \ar[r]
    & \mathbb Z
    }
\]

\begin{defn} (\cite{kottwitz-rapoport:alcoves}) \label{muperm}
Let $\mu = (1^{(r)}, 0^{(n-r)})$ be a dominant minuscule coweight. An
extended alcove $(x_0,\dots, x_{n-1})$ is called $\mu$-permissible, if for
all $i$:
\begin{enumerate}
\item
$\omega_i \le x_i \le \omega_i + \mathbf 1$ (where $\le$ is understood
component-wise),
\item
$\sum x_i := \sum_{j=1}^n x_i(j) = n-r-i$.
\end{enumerate}
\end{defn}


Given an alcove $x = (x_0, \dots, x_{n-1})$, we define
\[
r_{ij} = r_{ij}(x) = \sum_{k = j+1}^i (\omega_i(k) - x_i(k) + 1),\quad
i,j\in\{0,\dots, n-1\}
\]
where we identify $\{1,\dots, n\}$ with $\mathbb Z/n\mathbb Z$ and 
the sum is over the elements $j+1, j+2, \dots, i-1, i\in \mathbb
Z/n\mathbb Z$. (For instance, $r_{0,n-1} = \omega_0(n) - x_0(n) +1$.)
Since 
\[
x_i(j) = r_{ij} - r_{i,j-1} + \omega_i(j) + 1,
\]
we have

\begin{lm}
The alcove $x$  is uniquely determined by the tuple $(r_{ij}(x))_{i,j}$.
\end{lm}

Note that $r_{ii} = \sum \omega_i - \sum x_i + n = r$ is independent of
$x$, so that we can (and often will silently) omit the $r_{ii}$ from the
data.
Let us interpret the numbers $r_{ij}$ in terms of lattice chains. We know
that the set of extended alcoves is in bijection with the
extended affine Weyl group, and thus with the set of Iwahori orbits in
the affine flag variety. Explicitly, the alcove $(x_0,\dots, x_{n-1})$
corresponds to the $I$-orbit of the lattice chain $(\calL^x_0\subset \cdots
\subset \calL^x_{n-1})$
with
\[
\calL^x_i = \langle t^{x_i(1)}e_1, \dots, t^{x_i(n)}e_n \rangle
\]
(in fact, this lattice chain is the unique $T(\kappa[[t]])$-fixed point in
the corresponding $I$-orbit).

Then for any $\mu$-permissible alcove $x$ we have
\[
r_{ij}(x) = \dim_\kappa ( \calL^x_{i}+\stdt_{j'} / \stdt_{j'}),
\]
where $j' = j$ if $i > j$, and $j' = j-n$ otherwise. In fact,
\begin{eqnarray*}
\dim_\kappa ( \calL^x_{i}+\stdt_{j'} / \stdt_{j'}) & = &
\# \left\{ k\in \{1, \dots, n\};\ x_i(k) < \omega_{j'}(k) \right\}\\
& = & \# \left\{ k\in [j+1, i];\ x_i(k) < \omega_{j'}(k) \right\} =
r_{ij}(x),
\end{eqnarray*}
because $x_i(k) \ge \omega_i(k)$ for all $i$, $k$, and $\omega_i(k) =
\omega_{j'}(k)$ if (and only if) $k\not\in[j+1, i]$.
On the other hand, if we replace the chain $(\calL^x_i)_i$ by
another lattice chain in the same $I$-orbit, then the dimensions on the
right hand side of the above equation do not change, so we have

\begin{prop}
With respect to the natural bijection of Schubert cells (i.~e.~$I$-orbits
in $\Flag_{GL_n}$) and extended alcoves as above, the Schubert cell
corresponding to a $\mu$-permissible alcove $x$ is the locus of lattice
chains $(\calL_i)_i$ with
\[
\dim_\kappa ( \calL_{i}+\stdt_{j'} / \stdt_{j'}) = r_{ij}(x),
\]
for all $i$,$j$.
\end{prop}

It does not seem possible to give a simple characterization which tuples $(r_{ij})$
can occur (i.~e.~give a non-empty locus in the affine flag variety).

With the identification of the local model with a closed subscheme of the
affine flag variety above, we see that
\[
r_{ij} = \mathop{\rm rk} (\alpha_{j'-1}\circ \cdots \circ\alpha_i)(\mathcal
F_i),
\]
where the $\alpha_\bullet$ denote the transition maps in the local model.

It is clear that we can also use this characterization to describe the KR
stratification of the special fiber of a Shimura variety
corresponding to this local model. (See below
for an explicit formulation in terms of abelian varieties of the
corresponding numerical characterization of KR strata in the Siegel
case.)

\subsection{Numerical characterization of Schubert cells in the symplectic case}

In the symplectic case, we again use the description of alcoves as in
\cite{kottwitz-rapoport:alcoves}, adapted to our normalizations.

An extended alcove for the group $GSp_{2g}$ is an alcove $x$ for
$GL_{2g}$ which satisfies the duality condition:
\[
x_i(j) + x_{2g-i}(2g-j+1) = c-1 \text{ for all } i, j,
\]
for some $c= c(x)\in \mathbb Z$ depending only on $x$ (but not on $i$, $j$).
Here we set $x_{2g} = x_0 + \mathbf 1$.
In particular, the standard alcove $\omega$ satisfies this condition with
$c(\omega) = 0$. As above, let $\mu = (1,\dots, 1, 0, \dots, 0) =
(1^{(g)}, 0^{(g)})$. An alcove for $GSp_{2g}$ is $\mu$-permissible if it is
$\mu$-permissible for $GL_{2g}$ (for the same $\mu$, interpreted as a
coweight for $GL_{2g}$). If $x$ is a $\mu$-permissible alcove, then
$c(x)=1$.


As above, we define
\[
r_{ij} = r_{ij}(x) = \sum_{k = j+1}^i (\omega_i(k) - x_i(k) + 1),\quad
i,j\in\{0,\dots, 2g-1\}.
\]
For $0\le i\le g$, we have
\[
\omega_{2g-i}(k) - x_{2g-i}(k) + 1 = -\omega_{i}(2g-k+1) + x_i(2g-k+1)
- c(x) +1,
\]
so we can express all $r_{ij}$ in terms of $x_0,\dots, x_g$ (or,
alternatively, in terms of $x_g, \dots, x_{2g} (=x_0+\mathbf 1)$).
Of course, we again have, with notation as in the $GL_{2g}$-case,
\[
r_{ij}(x) = \dim \mathcal L^x_i + \lambda_{j'}/\lambda_{j'}.
\]
For $0<j<2g$ we define
\[
r_{2g,j}(x) := \dim \mathcal L^x_{2g} + \lambda_{j}/\lambda_{j} =
r_{0j}(x).
\]
We obtain that in the symplectic case
the $r_{ij}$ for $i = g, \dots, 2g$, $j=0, \dots, 2g-1$, $i\ne j$ ,
determine the alcove $x$ uniquely.

\begin{cor}
With respect to the natural bijection of Schubert cells (i.~e.~$I$-orbits
in $\Flag_{GSp_{2g}}$) and extended alcoves as above, the Schubert cell
corresponding to a $\mu$-permissible alcove $x$ is the locus of lattice chains $(\calL_i)_i$
with
\[
\dim_\kappa ( \calL_{i}+\stdt_{j'} / \stdt_{j'}) = r_{ij}(x),
\]
for all $i = g, \dots, 2g$, $j=0, \dots, 2g-1$, $i\ne j$.
\end{cor}

Finally, let us make these quantities explicit in terms of the
moduli space $\mathcal A_I$ of chains of abelian varieties.
Let $(B_i)_i \in \mathcal A_I(k)$. We denote by $\omega_i \subset
H^1_{DR}(B_i/k)$ the Hodge filtration, and by $^\perp$  the orthogonal
complement of a subspace of $H^1_{DR}(B_0)$ with respect to the pairing
induced by the (principal) polarization of $B_0$. Furthermore, we denote by
$\alpha_{ij} \colon H^1_{DR}(B_j) \rightarrow H^1_{DR}(B_i)$ the natural
map ($0 \le i\le j \le g$).

\begin{cor} \label{cor_num_char}
Let $x$ be a $\mu$-admissible alcove. 
With notation as above, the point $(B_i)_i$ lies in the KR stratum associated with $x$ 
if and only if for all $0 \le i, j\le g$:
\begin{eqnarray*}
&& \dim \omega_j / \alpha_{ji}(\omega_i) = g - r_{2g-i, 2g-j}(x) \text{ if } i
> j \\
&& \dim H^1_{DR}(B_i)/(\omega_i + \alpha_{ij}(H^1_{DR}(B_j))) =
j-i-r_{2g-i, 2g-j}(x) \text{ if } i< j \\
&& \dim \alpha_{0i}(\omega_i) + \alpha_{0j}(H^1_{DR}(B_j))^\perp = r_{2g-i,
j}(x) + j
\end{eqnarray*}
\end{cor}

\begin{proof}
Clearly, the above equalities determine all $r_{ij}(x)$ for $g \le i \le 2g$,
$0 \le j \le 2g-1$, $i\ne j$, so there is at most one $x$ such that they
are satisfied. Therefore we only have to check that the equalities above
translate to the description of the $r_{ij}$ in terms of lattice chains as
given above. Denote by $(\mathcal L_i)_i$ the lattice chain (over $k[[t]]$)
corresponding to the point $\omega_i \subset H^1_{DR}(B_i)$. To account for
the contravariance, here we identify the chain $(H^1_{DR}(B_i))_{i=0,\dots,
g}$ with the chain $(\lambda_{-i}/t\lambda_{-i})_{i=0, \dots, g}$, such
that $\omega_i$ gives rise to $t\lambda_{-i} \subset \mathcal L_{-i} \subset
\lambda_{-i}$. Using duality and periodicity, we can extend the chain
$(\mathcal L_i)_{i=-g, \dots, 0}$ to a ``complete'' chain $(\mathcal
L_i)_{i\in\mathbb Z}$.

We have, for $i>j$ (i.~e.~$g \le 2g-i < 2g-j \le 2g$),
\begin{eqnarray*}
\dim \omega_j / \alpha_{ji}(\omega_i) 
& = & \dim \mathcal L_{-j}/(\mathcal L_{-i} +  t\lambda_{-j})  \\
& = & \dim \mathcal L_{-j}/t\lambda_{-j}- \dim (\mathcal L_{-i} +
t\lambda_{-j})/t\lambda_{-j} \\
& = & g - \dim (\mathcal L_{2g-i} +
\lambda_{-j})/\lambda_{-j} \\
& = & g - r_{2g-i, 2g-j}(x),
\end{eqnarray*}
and for $i<j$ (i.~e.~$g \le 2g-j < 2g-i \le 2g$),
\begin{eqnarray*}
\dim H^1_{DR}(B_i)/(\omega_i + \alpha_{ij}(H^1_{DR}(B_j))) 
& = & \dim \lambda_{-i}/(\mathcal L_{-i} + \lambda_{-j}) \\
& = & \dim \lambda_{2g-i} / (\mathcal L_{2g-i} + \lambda_{2g-j}) \\
& = & \dim \lambda_{2g-i}/\lambda_{2g-j} - \dim (\mathcal L_{2g-i} +
\lambda_{2g-j})/\lambda_{2g-j} \\
& = & j-i-r_{2g-i, 2g-j}(x)
\end{eqnarray*}
and finally, for all $0 \le i, j \le g$ (i.~e.~$0\le j \le g \le 2g-i \le
2g$):
\begin{eqnarray*}
\dim \alpha_{0i}(\omega_i) + \alpha_{0j}(H^1_{DR}(B_j))^\perp
& = & \dim (\mathcal L_{-i} +t\lambda_0)/t\lambda_0 +
\lambda_{-2g+j}/t\lambda_0  \\
& = & \dim (\mathcal L_{-i} + \lambda_{j-2g})/\lambda_{j-2g} + \dim
\lambda_{j-2g}/t\lambda_0 \\
& = & \dim (\mathcal L_{2g-i} + \lambda_{j})/\lambda_{j} + j \\
& = & r_{2g-i, j}(x) + j.
\end{eqnarray*}
\qed
\end{proof}

This characterizes KR strata in $\mathcal A_I$ by the invariants defined
(and computed for $g\le 3$) in \cite{yu:krstrata}. In fact, the invariant
in the first line is $\sigma_{ji}$, in the second line we have
$\sigma'_{ij}$, and in the third line $d_{ij}$, with the notation of
loc.~cit. We discuss an explicit example:

\begin{eg}
Let $g=3$. We write down a couple of $\mu$-permissible alcoves, and compute
some of their invariants which appear in the corollary.

{\footnotesize
\begin{tabular}{ll}
$\tau$: &  (0,0,0,1,1,1), (0,0,0,0,1,1), (0,0,0,0,0,1), (0,0,0,0,0,0),
(-1,0,0,0,0,0), (-1, -1, 0,0,0,0) \\
$s_2s_3\tau:$ & (0,1,0,1,0,1), (0,0,0,1,0,1), (0,0,0,0,0,1),
(0,0,0,0,0,0), (-1, 0,0,0,0,0), (-1, 0,-1,0,0,0)\\
$s_1\tau:$ & (0,0,0,1,1,1), (0,0,0,0,1,1), (0,0,0,0,1,0), (0,0,0,0,0,0),
(0,-1,0,0,0,0), (-1,-1,0,0,0,0)\\
$s_0s_1\tau:$ & (0,0,0,1,1,1), (0,0,0,0,1,1), (-1,0,0,0,1,1),
(-1,0,0,0,0,1), (-1,-1,0,0,0,1), (-1,-1,0,0,0,0)\\
$s_2s_0s_1\tau:$ & (0,0,0,1,1,1), (0,0,0,1,0,1), (-1,0,0,1,0,1),
(-1,0,0,0,0,1), (-1,0,-1,0,0,1), (-1,0,-1,0,0,0)
\end{tabular}
}

So $\tau$ coincides with the base alcove $\omega_\bullet$ up to a shift,
and we obtain the other alcoves by applying the corresponding simple
reflections (using their description in (\ref{desc_si})). With these descriptions, it is
straight-forward to compute all the invariants $r_{ij}$, and also those in
the corollary above, for instance we obtain
\[
\begin{array}{|l|l|l|l|l|l|}
\hline
 & \sigma_{02} & \sigma'_{02} & \sigma_{03} & \sigma'_{03} & d_{12}\\\hline
\tau &           2 & 2 & 3 & 3 & 2 \\\hline
s_2s_3\tau &     2 & 1 & 3 & 2 & 3 \\\hline
s_1\tau &        2 & 2 & 3 & 3 & 2 \\\hline
s_0s_1\tau &     1 & 2 & 2 & 3 & 2 \\\hline
s_2s_0s_1\tau &  1 & 2 & 2 & 3 & 3 \\\hline
\end{array}
\]
\end{eg}
As claimed, these values agree with those obtained in
\cite{yu:krstrata}.

\section{The minimal KR stratum}
\label{sec:03}

\subsection{The unitary group.}
\label{unit_gp}
Let
$x_0=(A_0,\lambda_0)$ be a $g$-dimensional superspecial principally
polarized 
abelian variety over $k$. Denote by $G_{x_0}$ the automorphism group
scheme 
over $\Z$ associated to $x_0$; for any commutative ring $R$, the
group of its $R$-valued points is
\begin{equation}\label{eq:31}
  G_{x_0}(R)=\{\,x\in (\End(A_0)\otimes R)^\times\ ;\ x' x=1\, \},
\end{equation}
where $x\mapsto x'$ is the Rosati involution induced by $\lambda_0$.
Let $(M_0,\<\, ,\>_0)$ be the \dieu module of $x_0$. Let 
\begin{equation}
  \wt M_0:=\{ m\in M_0\ ; \ F^2 m+p m=0\, \}
\end{equation}
be the skeleton of $M_0$; this is a \dieu module over $\F_{p^2}$
together with a quasi-polarization induced from $M_0$, and
one has $\wt M_0\otimes_{W(\F_{p^2})} W(k)=M_0$. We can choose a basis
$e_1, \dots e_{2g}$ (in $\wt M_0$) for $M_0$ such that
\begin{equation}\label{eq3.3}
  Fe_{g+i}=-e_i \quad \text{and } \ F(e_i)=p e_{g+i}, \quad
  \forall\, i=1,\dots, g.
\end{equation}
and the representing matrix for
$\<\, ,\>_0$ is
\[
\begin{pmatrix}
  0 & \wt I_g \\
  -\wt I_g & 0
\end{pmatrix},\quad \wt I_g=\text{anti-diag(1,\dots,1)}.\]

Let $V_0:=\wt M_0/V \wt M_0$ and ${\rm pr}\colon W(k)\to k$ be the
natural map. Define $\varphi_0(x,y):=\<x,Fy\>_0$ and $\ol
\varphi_0:={\rm pr}(\varphi_0)$ on $M_0$.
The pairing $\ol \varphi_0$ induces a non-degenerate
Hermitian $\F_p$-bilinear form
(again denoted by)
\[ \ol \varphi_0:V_0\times V_0\to \F_{p^2}. \]
Indeed, using the basis $\{e_{g+1},\dots, e_{2g}\}$ for $V_0$, the pairing
$\varphi_0$ is simply given by
\[ \ol \varphi_0((a_i), (b_{i}))=\sum_{i} a_i \bar b_{g+1-i}. \]
Denote by $\ol {\bfG}_0$ the unitary group $U(V_0,\ol \varphi_0)$ over
$\F_p$. Note that the group $\Aut(M_0/VM_0, \varphi_0)$
consisting of elements $h\in
\Aut(M_0/VM_0)$ that preserve the pairing $\varphi_0$ is not
 $\ol{\bfG}_0(k)$ but rather $\ol {\bfG}_0(\Fp)$, a finite group.
Indeed, using the basis $\{e_{g+1},\dots, e_{2g}\}$, one shows
that the group $\Aut(M_0/VM_0, \varphi_0)$ is the group of all $h \in
GL_g(k)$ such that $h^t\cdot\wt I_g h^{(p)}=\wt I_g$, and since these $h$
automatically lie in $GL_g(\mathbb F_{p^2})$, it 
is precisely $\ol {\bfG}_0(\Fp)$.

\subsection{}

Let $B_p$ be the quaternion division algebra over $\Qp$ and $O_{B_p}$
be its ring of integers. We choose a presentation
\[ O_{B_p}=W(\F_{p^2})[\Pi], \quad \Pi^2=-p \text{ \  and \ } \Pi\, a
= a^\sigma \Pi, \quad \forall\, a\in W(\F_{p^2}). \] We can write
$\wt M_0=\oplus_{i=1}^g O_{B_p} f_i$, where $f_i=e_{g+i}$ and $\Pi$ acts
by the Frobenius $F$. We have $\Pi^*=V=-F=-\Pi$, where $*$ is the
canonical involution. One easily
computes
\[ (\alpha+\beta\Pi)^*=\alpha^\sigma+\Pi^*
\beta^*=\alpha^\sigma-\beta \Pi. \] By Tate's theorem on 
homomorphisms of abelian varieties and \dieu modules, we have
the identifications
\begin{equation}
  G_{x_0}(\Z_p)=\Aut_{\rm DM}(M_0,\<\, ,\>_0)=\Aut_{O_{B_p}}(\wt
  M_0,\<\, ,\>_0).
\end{equation}
We also have
\begin{equation}
  G_{x_0}\otimes \Fp=\Aut_{O_{B_p}}(\wt
  M_0\otimes_{\Zp} \Fp,\<\, ,\>_0),
\end{equation}
regarded as algebraic groups over $\Fp$. Since the subspace $V\wt
M_0/p \wt M_0$ is stable under the action of $G_{x_0}\otimes \Fp$,
we have a homomorphism of algebraic groups
\begin{equation}
  \rho:G_{x_0}\otimes \Fp\to \Aut(V_0).
\end{equation}
The following lemma is easily verified.

\begin{lm} \label{31} \

{\rm (1)} One has $\varphi_0(y,x)=\varphi_0(x,y)^\sigma$ for $x,y\in \wt
M_0$.

{\rm (2)} If $a\in W(\F_{p^2})$, then $\varphi_0(ax, y)=
\varphi_0(x,a^*y)$ for $x, y\in \wt M_0$. If  $a\in W(\F_{p^2})\Pi$,
then $\varphi_0(ax, y)= \varphi_0(x,a^*y)^\sigma$ for $x, y\in \wt
M_0$. Consequently, we have
\[ \tr_{W(\F_{p^2})/\Zp}(ax,y)=\tr_{W(\F_{p^2})/\Zp}(x,a^*y), \quad
\forall \, a\in
  O_{B_p}\text{\ and\ }  x, y\in \wt M_0.\]

{\rm (3)} An element $h$ in $\Aut_{O_{B_p}}(\wt M_0)$ preserves
the pairing $\<\, , \>_0$ if and only if it preserves the pairing
$\varphi_0$. Consequently, the homomorphism $\rho$ factors through
the subgroup $\ol {\bfG}_0$.

{\rm (4)} The homomorphism $G_{x_0}(\Z_p)=\Aut_{\rm DM}
(M_0,\<\, ,\>_0)\to \ol {\bfG}_0(\Fp)$ is surjective.

{\rm (5)} The homomorphism $\rho$ induces an isomorphism $\rho:
(G_{x_0}\otimes \Fp)^{\rm red} \simeq  \ol {\bfG}_0$, where
$(G_{x_0}\otimes \Fp)^{\rm red}$ denotes the maximal reductive
quotient of $G_{x_0}\otimes \Fp$
\end{lm}

\subsection{}
Let $d$ be an integer with $0\le d\le g$, and let $M$ be a \dieu
module over $k$ with
\[ VM_0\subset M\subset M_0, \quad \text{and}\quad \dim_k M/VM_0=d. \]
The subspace $M/VM_0$ defines an element in the Grassmannian ${\rm
  Gr}(d;V_0)(k)$ of $d$-dimensional subspaces of $V_0$.

\begin{lm}\label{32}
  Notation as above, the \dieu module $M$ is superspecial
  if and only if  $M/VM_0\in {\rm Gr}(d;V_0)(\F_{p^2})$.
\end{lm}

\begin{proof}
  See \cite[Lemma 6.1]{yu:ss_siegel}.
\qed
\end{proof}

\subsection{} \label{subsec34}
Let $g_\tau=
\begin{pmatrix}
  0 & I_g \\ -p I_g & 0
\end{pmatrix}$ be a representative in $\GSp_{2g}(W(k))$ for the double
coset corresponding to $\tau\in \wt W$. 
This gives rise to a point in the local model $M^{\rm
  loc}_I(k)$ described as follows (where $\ol\stdp_{i} = \stdp_i/p\stdp_i$):


\[
\begin{CD}
 \ol \stdp_{-g} @>>>\ \cdots \ @>>>  \ol \stdp_{-1} @>>>   \ol
 \Lambda_{0} \\
 \cup & & {}  & & \cup & & \cup \\
 \ol \calL_{-g} @>>>\ \cdots \ @>>>  \ol \calL_{-1} @>>>   \ol
 \calL_{0}, \\
\end{CD} \]
with
\[ \ol \calL_0=\<e_1,\dots, e_g\>,\quad \ol \calL_{-1}=\<e_1,\dots, e_{g-1},
e_{2g}\>,\]
\[  \ol \calL_{-2}=\<e_1,\dots, e_{g-2}, e_{2g-1}, e_{2g}\>,\
\dots,\ \ol \calL_{-g}=\<e_{g+1},\dots , e_{2g}\>. \]

Here, in the description of $\calL_{-i}$, $e_1,\dots, e_{2g}$ is the
standard basis of $\ol\stdp_{-i}$. In other words, we can define
$\calL_{-i}$ as the image of $\ol\stdp_{-g-i}$ in $\ol\stdp_{-i}$.
Denote by $\alpha_{i,j}: \stdp_{-j}\to \stdp_{-i}$ the
composition.
By duality we can extend the lattice chain $(\calL_{-i})_{i = 0,\dots, g}$ to
a complete periodic lattice chain $(\calL_i)_{i\in \mathbb Z}$. We then
have
\begin{equation}\label{eq:37}
\alpha_{i,g+i}(\ol \calL_{-g-i})=0,
\quad \forall\, i=0,\dots, g.
\end{equation} 

We see that
\begin{equation}\label{eq:35}
\alpha_{0,i}(\ol \stdp_{-i})^\bot=\alpha_{0,i}(\ol \calL_{-i}),
\quad \forall\, 
i=0,\dots, g,
\end{equation}
where $\bot$ stands for the orthogonal complement with respect to $\psi_0$.
Note that $\alpha_{0,i}(\ol \stdp_{-i})^\bot=\alpha_{0,g+i}(\ol
\Lambda_{-g-i})$. The condition (\ref{eq:35}) is equivalent to
\begin{equation}\label{eq:36}
\alpha_{i,g+i}(\ol \stdp_{-g-i})=\ol \calL_{-i},
\quad \forall\, i=0,\dots, g.
\end{equation} 
Conversely, the condition (\ref{eq:36}) or (\ref{eq:35}) 
characterizes whether a
point $(\ol \calL_\bullet)$ of $M^{\rm loc}_I(k)$ lies in the
minimal stratum.

\begin{prop}\label{33} 
  Let $x=(\ul A_0\to \cdots \to \ul A_g)\in \calA_I(k)$ be a geometric
  point and $(M_\bullet)=(M_{-g}\subset M_{-g+1}\subset \dots \subset
  M_0)$ be the corresponding chain of \dieu modules. Let $\ol
  {M}_{-i}:=M_{-i}/pM_0$ for $i=0,\dots, g$. Then $x\in \calA_\tau$ if and
  only if
  \begin{equation}
    \label{eq:38}
  \<\ol M_{-1}, F\ol M_{-g+1}\>_0=  \<\ol M_{-2}, F\ol M_{-g+2}\>_0
  =\dots =\<\ol
  M_{-g+1}, F\ol M_{-1}\>_0=0.
  \end{equation}
\end{prop}

\begin{proof}
We choose an isomorphism $M_\bullet \simeq
  \stdp_\bullet \otimes W(k)$ compatible with polarizations. We have
\[ 0=\ol M_{0}^{\bot}\subset \ol M_{-1}^{\bot} \subset\dots \subset \ol
M_{-g-1}^{\bot} \subset \ol M_{-g}^{\bot}=\ol M_{-g} \subset \dots \subset
\ol M_{0} \]
The condition (\ref{eq:35}) says that 
\[ V \ol M_{-i}=\ol M_{-g+i}^{\bot},\quad \forall\, i=0,\dots, g. \]
It follows from the discussion above
that $x\in \calA_\tau$ if and
only if the condition
   \begin{equation}
    \label{eq:39}
  \<\ol M_{-g+1}, V\ol M_{-1}\>_0=  \<\ol M_{-g+2}, V\ol M_{-2}\>_0
  =\dots =\<\ol M_{-1}, V\ol M_{-g+1}\>_0=0.
  \end{equation}
holds. The condition (\ref{eq:39}) is the same as the condition
(\ref{eq:38}). This proves the proposition. \qed
\end{proof}

\begin{lm}\label{36}
  Let $x=(\ul A_\bullet)\in \calA_I(k)$ be a geometric point and let
  $(M_\bullet)$ be the chain of associated \dieu modules.  
  Fix an element $i\in I$. Suppose that there is an isomorphism $\xi:
  M_\bullet \simeq \stdp_\bullet \otimes W(k)$, compatible with the
  polarizations, such that $\xi(V \ol M_{-i})=\ol {\calL}_{-i}$ and $\xi(V
  \ol M_{-g+i})=\ol {\calL}_{-g+i}$, where $\ol \calL_\bullet \subset \ol
  \stdp_\bullet$ is the point $g_\tau\in {\bf M}^{\rm loc}_I(\Fp)$. Then
  both $M_{-i}$ and $M_{-g+i}$ are superspecial.   
\end{lm}
\begin{proof}
  Let $M_{-g-i}$ be the dual \dieu module of $M_{-g+i}$ 
  with respect to the
  pairing $\frac{1}{p}\<\, ,\>_0$. It follows from (\ref{eq:36}) that
  $M_{-g-i}=VM_{-i}$. It follows from (\ref{eq:37}) that 
  $VM_{-g-i}=pM_{-i}$. This shows $V^2M_{-i}=pM_{-i}$. Therefore,
  $M_{-i}$ is superspecial. The same argument shows that $M_{-g+i}$ is
  also superspecial. \qed
\end{proof}

\begin{thm}\label{34}
  Let $x=(\ul A_\bullet)\in \calA_I(k)$ be a geometric point and let
  $(M_\bullet)$ be the chain of associated \dieu modules. Then $x\in
  \calA_\tau$ if and only if
  \begin{enumerate}
  \item[{\rm (i)}] each $M_i$ is superspecial, and
  \item[{\rm (ii)}] the subspace $\wt M_{-i}/V \wt M_0\subset V_0$ is
    isotropic 
  with respect to $\ol \varphi_0$ for $i\ge \lceil g/2 \rceil$ and
  $\wt M_{-i}/V \wt M_0=  (\wt M_{-g+i}/V \wt M_{0})^{\bot}$ 
  with respect to  $\ol \varphi_0$ for  $i< \lceil g/2 \rceil$.
  \end{enumerate}
\end{thm}

\begin{proof}
If every $M_i$ is superspecial, then the condition (ii) is equivalent
to the condition (\ref{eq:38}). 
On the other hand, the lemma above shows that for $x =(A_i)_i \in \mathcal
A_\tau$, all the Dieudonn\'e modules of the $A_i$ are superspecial.
\qed 
\end{proof}

\begin{remark}
  We have the following variant of the characterization of the minimal KR
  stratum in the moduli space $\calA_I$. Let $x$
  and $(M_\bullet)$ be as in Theorem~\ref{34}. We extend the chain of
  \dieu modules $(M_\bullet)$ to $(M_{-i})_{0\le i\le 2g}$ using
  duality (as in the proof of Lemma~\ref{36}). Then $x\in \calA_\tau$
  if and only if $FM_{-i}=M_{-g-i}$ for all $i=0,\dots,g$.  
\end{remark}

We denote by $\Lambda_{g,1,N}$ the set of superspecial
points in the moduli space $\calA_{g,1,N}(k)$.

\begin{cor}\label{35} \

{\rm (1)} Let $x, x'$ be two points in $\calA_\tau$ and let
$(M_\bullet), (M'_\bullet)$ be the corresponding chains of \dieu modules. Then we
have an isomorphism $(M_\bullet)\simeq (M'_\bullet)$ as chains of \dieu
modules with quasi-polarizations.

{\rm (2)} We have $\# \calA_\tau(k) =\# \Lambda_{g,1,N}\cdot \#(\ol {\bf
  G}_0/B_0)(\Fp)$, where $B_0$ is any Borel subgroup of
  $\ol {\bf G}_0$ over $\Fp$.
\end{cor}
\begin{proof}
  (1) Since $(M_0,\<\, ,\>_0)$ and $(M'_0,\<\, ,\>'_0)$ are isomorphic
      (Theorem~\ref{34}), we can assume that $M_0=M_0'$ and
      $M_{-g}=M_{-g}'$. 
      Since the group $\ol {\bf G}_0(\Fp)$ acts transitively
      on the space of maximal chains of isotropic 
      subspaces of $V_0$ with
      respect to $\ol \varphi_0$, there is an $\ol h\in \ol {\bf
      G}_0(\Fp)$ such that $\ol h (M_{-i}/M_{-g})=(M'_{-i}/M_{-g})$ 
      for all $i$
      (Theorem~\ref{34}). Since the map $\Aut_{\rm DM}
      (M_0,\<\, ,\>_0)\to \ol {\bfG}_0(\Fp)$ is surjective
      (Lemma~\ref{31} (4)), there is an element $h\in \Aut_{\rm DM}
      (M_0,\<\, ,\>_0)$ such that $h (M_{-i})=(M'_{-i})$ for all $i$.

  (2) This follows immediately from (1) and Theorem~\ref{34}.\qed
\end{proof}

There is an explicit formula for the number $\# (\ol {\bf G}_0/B_0)(\mathbb
F_p)$, in fact, slightly more generally:

\begin{lm}\label{39}\
We have, for any $p$-power $q$, 
\[
\# (\ol {\bf G}_0/B_0)(\mathbb F_q) = \prod_{i=1}^g
\frac{1-(-q)^i}{1-(-1)^iq} = \left\{ \begin{array}{ll}
\prod_{i=1}^d
\frac{(q^{2i}-1)(q^{2i-1}+1)}{(q^2-1)} & \text{if } g=2d\text{ is even,}
\\[.3cm]
\prod_{i=1}^d
\frac{(q^{2i}-1)(q^{2i+1}+1)}{(q^2-1)} & \text{if } g=2d+1\text{ is odd.}
\end{array}\right.
\]
\end{lm}

\begin{proof}
This is just a computation about the unitary group over a finite field,
which we omit here.  The result can also be extracted from the general
theorems in Carter's book \cite{carter}, Chapter 14.
\qed
\end{proof}

Using the mass formula for $\#\Lambda_{g,1,N}$ due to Ekedahl and
Hashimoto-Ibukiyama (cf. \cite[Section 3]{yu:ss_siegel})
\begin{equation}
\#\Lambda_{g,1,N}=\#\Sp_{2g}(\Z/N\Z) \cdot
\frac{(-1)^{g(g+1)/2}}{2^g}  \prod_{i=1}^g [\zeta(1-2i)
  (p^i+(-1)^i)]   
\end{equation}
we get
\begin{prop}\label{310}
  \begin{equation}
    \label{eq:312}
   \#\calA_\tau(k)=\#\Sp_{2g}(\Z/N\Z)\cdot
\frac{(-1)^{g(g+1)/2}}{2^g}  \prod_{i=1}^g [\zeta(1-2i)
  (p^i+(-1)^i)]\cdot L_p 
  \end{equation}
where 
\[ L_p =
\left\{ \begin{array}{ll}
\prod_{i=1}^d
\frac{(p^{2i}-1)(p^{2i-1}+1)}{(p^2-1)} & \text{if } g=2d\text{ is even,}
\\[.3cm]
\prod_{i=1}^d
\frac{(p^{2i}-1)(p^{2i+1}+1)}{(p^2-1)} & \text{if } g=2d+1\text{ is odd.}
\end{array}\right.
\]
\end{prop}

One can give similar formulas in the case of arbitrary parahoric level
structure, see \cite{goertz-yu:mass}. In Section 6, we will give similar
formulas for the numbers of connected components of some other KR strata
(which are contained in the supersingular locus).\\

\section{Supersingular KR-strata}
\label{sec:superspecial_strata}

\subsection{}

Recall that the affine Weyl group $W_a$ (of $G = GSp_{2g}$) is generated by
the simple affine reflections $s_0, \dots, s_g$. See (\ref{desc_si}) for an
explicit description.
Denote
by $\tau\in \Omega$ the length $0$ element of $\wt W$ with $t^\mu \in
W_a\tau$. We can represent it by the matrix $g_\tau$ as in Section
\ref{subsec34} (or rather its analogue over $k[[t]]$).
We use the notation for extended alcoves as introduced in Section
2, and identify the extended affine Weyl group $\wt W$ with the set of
extended alcoves. For instance, we have $\tau = (\tau_i)_i$ with $\tau_i =
(0^{(g+i)}, 1^{(g-i)})$ for $i=0,\dots, g$, $\tau_i =
(-1^{(i-g)}, 0^{(2g-(i-g))})$ for $i=g,\dots, 2g$.

\begin{defn}
For $0 \le i \le [\frac{g}{2}]$, denote by $W_{\{i, g-i\}}$ the subgroup of $W_a$
generated by all simple reflections excluding $s_i$ and $s_{g-i}$. 
\end{defn}

\begin{lm}
The map $w \mapsto w\tau$ yields a bijection
\[
W_{\{i,g-i\}} \isoto \{ x \in {\rm Adm}(\mu);\ x_i = \tau_i,\ x_{g-i} = \tau_{g-i}
\}.
\]
\end{lm}

\begin{proof}
It is easy to see that the map $w \mapsto w\tau$ induces a bijection
\[
W_{\{i,g-i\}} \isoto \{ x \in \wt W; \ x_i = \tau_i,\ x_{g-i} = \tau_{g-i}
\}.
\]
Hence it only remains to show that all elements $x$ of the set on the
right hand side are in fact $\mu$-admissible. We need to check that
$\omega_j \le x_j \le \omega_j+\mathbf 1$ for all $j$. Now we have
$\omega_j \le \tau_i$ whenever $j \ge i-g$, and $\tau_i \le
\omega_j+\mathbf 1$ whenever $j \le i+g$. Since $x_i = \tau_i$, and
$x_{g-i} = \tau_{g-i}$ by assumption, we have
\[
\omega_j \le \tau_i \le x_j \le \tau_{-i} \le \omega_j+\mathbf 1
\]
for $-i \le j \le i$, and we have similar inequalities for $i \le j \le
g-i$, $g-i\le j\le g+i$ etc.
\qed
\end{proof}

\begin{defn}\label{43}
\begin{enumerate}
\item
We call a KR stratum \emph{supersingular}, if it is contained in the
supersingular locus.
\item
For $0 \le i \le [\frac{g}{2}]$, we call a KR stratum
\emph{$i$-superspecial} if for all $k$-valued points $(A_i)_i$ in the
concerning stratum, $A_i$ and $A_{g-i}$ are superspecial, and the isogeny
$A_i \rightarrow A_{g-i}^\vee$ is isomorphic to the Frobenius morphism
$A_i\rightarrow A_i^{(p)}$, i.~e.~if there is a commutative diagram
\[
 \xymatrix{ A_i
    \ar[d]^{=} \ar[r] & A_{g-i}^\vee \ar[d]^{\cong} \\ A_i \ar[r]^{F} &
    A_i^{(p)} }
\]
\item
We call a stratum \emph{superspecial}, if it is $i$-superspecial for some
$i$.\end{enumerate}
\end{defn}
 
\begin{prop}
The KR stratum associated with $x \in {\rm Adm}(\mu)$ is $i$-superspecial if
and only if $w \in W_{\{i,g-i\}}\tau$.

In particular, for all $w \in \bigcup_i W_{\{i,g-i\}}$, the KR stratum
associated with $w\tau$ is supersingular.
\end{prop}

\begin{proof}
Let $w \in W_{\{i,g-i\}}$. The lemma above gives us that $(w\tau)_i = \tau_i$,
$(w\tau)_{g-i} = \tau_{g-i}$. Since the lattices of the standard
lattice chains are 
fixed by $\Iw$, this means that for all chains $(\mathcal L_j)_j$ in
the Schubert cell 
corresponding to $w\tau$, we have $\mathcal L_i = \Lambda_{i-g}$, $\mathcal
L_{g-i} = \Lambda_{-i}$. Now Lemma \ref{36} implies that the KR stratum
associated with $w\tau$ is $i$-superspecial.

On the other hand, suppose $w\in \mathop{\rm Adm}(\mu)$ gives rise to an
$i$-superspecial stratum.
It follows that for all lattice chains $(\mathcal L_j)_j$ in the
corresponding Schubert cell, $\mathcal L_i = \mathcal \stdt_{i-g} =
\tau\stdt_i$ and $\mathcal L_i =\tau\stdt_i$. 
Hence the preceding lemma yields the proposition.
\qed
\end{proof}

In \cite{goertz-yu:kreo}, we prove that every KR stratum which is
entirely contained in the supersingular locus, is superspecial. The proof
relies on exploiting the relationship between the KR stratification and the
Ekedahl-Oort stratification of $\mathcal A_g$.

\begin{thm}\label{thm_ss_KR_strata} (\cite{goertz-yu:kreo} Cor.~7.4)
If $x\in {\rm Adm}(\mu)$ gives rise to a supersingular KR stratum, then $x$
lies in $\bigcup_i W_{\{i,g-i\}}\tau$, i.~e.~the stratum corresponding to
$x$ is superspecial.
\end{thm}

This theorem can also be understood as a statement about the non-emptiness
of certain affine Deligne-Lusztig varieties: for all $x \in {\rm
Adm}(\mu)\setminus \bigcup_i W_{i,g-i}\tau$, there exists a
$\sigma$-conjugacy class $[b]$ different from the supersingular class, such
that $X_x(b) \ne \emptyset$. See \cite{haines:clay}
Prop.~12.6 and \cite{ghkr} 5.10. From this point of view, one can check the
corresponding statement in the function field case (using a computer program which
evaluates foldings of galleries; cf.~loc.cit.) for $g \le 4$. With the
algorithms known to us, the case $g=5$ is out of reach.

Finally, we note the following proposition, which in particular gives a
lower bound on the dimension of the supersingular locus in $\mathcal A_I$.
In all cases where we know the latter dimension, this bound turns out to be
sharp. The bound also shows that the codimension of the supersingular locus
is much smaller in the Iwahori case than in the case of good reduction,
i.~e.~the case of $\mathcal A_g$.

\begin{prop}\label{dim_sspKR}
The dimension of the union of all superspecial KR strata is $g^2/2$, if $g$
is even, and $g(g-1)/2$, if $g$ is odd. There is a unique superspecial
stratum of this maximal dimension.
\end{prop}

As pointed out in the introduction, comparing this dimension with the
dimension of the $p$-rank $0$ locus, one can prove that for all even $g$
(and also for $g=1$) the supersingular locus and the union of all
superspecial KR strata have the same dimension.

One should note however that this locus is not at all equidimensional. In
fact, it has $[\frac g2]$ maximal superspecial KR strata, corresponding to
the longest elements of the Weyl groups $W_{\{i,g-i\}}$. The supersingular
locus is not equidimensional either, in general (and even for $g=2$).

\begin{proof}
Let $0\le i \le [\frac g2]$. The Weyl group $W_{\{i,g-i\}}$ is isomorphic
to the product of two copies of the Weyl group of the symplectic group
$Sp_{2i}$, and one copy of the Weyl group of $SL_{g-2i}$. For the former
groups, the longest element has length $i^2$, for the latter one it has
length $(g-2i)(g-2i-1)/2$, so the longest element of $W_{\{i,g-i\}}$ has
length 
\[
2i^2 + (g-2i)(g-2i-1)/2 = (2i - \frac{2g-1}{4})^2 + \frac{4g^2-4g-1}{16}.
\]
This parabola has its global minimum at $i = \frac{g}4 - \frac 18$, so
restricted to 
the set $\{0, 1, \dots, [\frac g2] \}$, it takes its maximum at $i=0$
if $g$ 
is odd, and at $i=[\frac g2] = \frac g2$, if $g$ is even. Correspondingly,
the maximum value is $g(g-1)/2$ if $g$ is odd, and $g^2/2$ if $g$ is even.
\qed
\end{proof}

In \cite{goertz-yu:kreo}, Thm.~8.8, we prove that the dimension of the
$p$-rank $0$ locus is $[\frac{g^2}{2}]$.
We also show that every irreducible component of maximal dimension of the
union of all superspecial KR strata is actually an irreducible component of
the $p$-rank $0$ locus, and hence in particular an irreducible component of
the supersingular locus. Furthermore, if $g$ is even, then every
top-dimensional irreducible component of the supersingular locus is of this
form.

\section{Deligne-Lusztig varieties}
\label{sec:05}
\subsection{Reminder on Deligne-Lusztig varieties}

We first introduce the notation used in this section.
Let $G$ be a connected reductive group over a finite field $\mathbb F_q$.
We fix an algebraic closure $k$ of $\mathbb F_q$. Let $T \subset G$ be a
maximal torus defined over $\mathbb F_q$, and $B$ a Borel subgroup of $G$
defined over $\mathbb F_q$ and containing $T$. Denote by
$W$ the Weyl group $N_GT(k)/T(k)$. Let $\sigma$ denote the Frobenius $x
\mapsto x^q$ on $k$, and also the Frobenius on $G(k)$. Below we often
silently identify $G$, $B$, $G/B$, etc.~with their sets of $k$-valued
points.

We consider the \emph{relative position map}
\[
\mathop{\rm inv} \colon G/B \times G/B \longrightarrow W,
\]
which maps a pair $(g_1, g_2)$, $g_1, g_2\in G(k)$ to the
unique element $w$ such that $g_1^{-1} g_2 \in BwB$.
(With a little more
effort, introducing \emph{the} Weyl group of $G$ (as the projective limit
over all isomorphisms between Weyl groups for pairs $(T,B)$ as above), one
can make this independent of the choice of a torus, in some sense. For us,
working with a fixed torus is good enough, however.) We recall the definition
of Deligne-Lusztig varieties.

\begin{defn} (\cite{deligne-lusztig})
Let $w \in W$. The Deligne-Lusztig variety associated to $w$ is
\[
X(w) = \{ g \in (G/B)(k);\ \mathop{\rm inv}(g, \sigma g) = w  \}.
\]
\end{defn}

Then $X(w)$ is a locally closed subvariety of $G/B$, which is smooth
of pure dimension $\ell(w)$, the length of $w$.

\subsection{``Local model diagram'' for Deligne-Lusztig varieties}
\label{locmod_for_DL}

We use the notation of the previous section. Consider the diagram
\[
G/B \stackrel{\pi}{\longleftarrow} G \stackrel{L}{\longrightarrow} G/B,
\]
where $\pi$ is the projection, and $L$ is the concatenation of the Lang
isogeny $G \rightarrow G$, $g\mapsto g^{-1}\sigma(g)$, with the projection.
Both maps in this diagram are smooth, of the same relative dimension, and
under these maps, Deligne-Lusztig varieties and Schubert cells correspond
to each other: $\pi^{-1}(X(w)) = L^{-1}(BwB/B)$. For example, this implies
instantly that $X(w)$ is smooth of pure dimension $\ell(w)$.

In particular, we see that the singularities of the closure of $X(w)$ are
smoothly equivalent to the singularities of the Schubert variety
$\overline{BwB/B}$. This gives a simpler approach to some of the results of
Hansen \cite{hansen}.

\subsection{Connected components of Deligne-Lusztig varieties}

There is the following result by Lusztig (unpublished) about the
irreducibility (or, equivalently, connectedness) of Deligne-Lusztig
varieties; see the papers \cite{digne-michel} by Digne and Michel
(Prop.~8.4), or \cite{bonnafe-rouquier:irred} by Bonnaf\'e and Rouquier
(Theorem 2) where also the more general case of Deligne-Lusztig varieties in
$G/P$ for a parabolic subgroup $P\subset G$ is considered.  Let $S\subset W$
denote the set of simple reflections (for our choice of Borel group $B$).
For any subset $J \subseteq S$, we have the standard parabolic subgroup
$W_J \subseteq W$ generated by the elements of $J$. (Note that the notation
here, which is the usual one, differs from the notation $W_{\{i,g-i\}}$
used in the previous section.) The Frobenius $\sigma$ acts on the Weyl
group $W$.

\begin{fact}\label{52}
Let $w\in W$. The Deligne-Lusztig variety $X(w)$ is irreducible if and only
if $w$ is not contained in any proper $\sigma$-stable standard parabolic
subgroup of $W$.
\end{fact}

The following corollary gives the number of irreducible components of an
arbitrary $X(w)$; it follows easily from the results in
\cite{bonnafe-rouquier:irred}. 

\begin{cor} \label{num_conn_comp_of_DL}
Let $w\in W$, and let $W_J$, $J\subseteq S$, be the minimal $F$-stable
standard parabolic subgroup of $W$ which contains $w$. Let $P_J = BW_JB$ be
the associated parabolic subgroup. Then the number of irreducible
components of $X(w)$ is $\# (G/P_J)(\mathbb F_q)$.
\end{cor}

In the corollary, we allow $J=S$, in which case $P_J = G$ (and in fact
$X(w)$ is irreducible by the theorem).

\begin{proof}
As in the proof of the ``only if'' half of
\cite{bonnafe-rouquier:irred}, Theorem 2, we
consider the projection $\mathbf p \colon G/B \rightarrow G/P_J$. It
is easy to see (see loc.~cit.) that
\[
\mathbf p(X(w)) = X_J(1) := \{ g \in G/P_J(k);\ g^{-1}\sigma(g) \in
P_J \}. 
\]
Note that this proves that $X(w)$ is not connected unless $J=S$.

To prove the corollary, we need to show that the fibers of the
restriction $X(w) 
\rightarrow X_J(1)$ of $\mathbf p$ are connected. Because this morphism is
equivariant under the $G(\mathbb F_q)$-actions on both sides, it is enough
to consider the fiber over $1$. Let $\pi_J \colon P_J \rightarrow
P_J/R_u(P_J) =:M_J$ be the projection to the maximal reductive quotient of
$P_J$. The image of $T$ and $B$ under $\pi_J$ are a maximal torus and a
Borel subgroup of $M_J$. The Weyl group of $M_J$ (with respect to this
maximal torus) can be naturally identified with $W_J$. In particular, we
can consider $w$ as an element of the Weyl group of $M_J$. Since $P_J/B =
M_J/\pi_J(B)$, we have
\begin{eqnarray*}
(\mathbf p|_{X(w)})^{-1}(1) & = & \{ g \in P_J/B;\ g^{-1}\sigma(g) \in
BwB \} \\ 
& \cong & \{ g \in M_J/\pi_J(B);\ g^{-1}\sigma(g) \in \pi_J(B) w \pi_J(B)
\}
\end{eqnarray*}
Since by the definition of $J$, $w$ is not contained in any proper
$F$-stable standard parabolic subgroup of $W_J$, Fact~\ref{52} implies that
this fiber is irreducible, as we had to show.
\qed
\end{proof}

\subsection{Affineness of Deligne-Lusztig varieties}
\label{DL_affine}
We conclude by recalling some results about the affineness of
Deligne-Lusztig varieties. Haastert has shown that every $X(w)$ is
quasi-affine, see \cite{haastert:diplom} Satz 2.3.
This is proved by constructing an ample line bundle on $G/B$ whose
restriction to $X(w)$ is trivial. Deligne and Lusztig have given a
criterion of the affineness of $X(w)$ in terms of the underlying root
system. 
This implies in particular that $X(w)$ is
affine whenever the cardinality $q$ of the residue class field is greater
or equal than the Coxeter number of $G$; see \cite{deligne-lusztig}
Thm.~9.7. For further results in this direction see the recent papers by 
Orlik and Rapoport \cite{orlik-rapoport}, He \cite{he}, and by Bonnaf\'e
and Rouquier \cite{bonnafe-rouquier:a}.

\section{Geometric structure of supersingular KR strata}
\label{sec:structure_ssp_strata}

\subsection{Supersingular KR strata are disjoint unions of DL varieties}

Fix a point $(A_i)_i$ in the minimal KR stratum. We denote by
$G'$ the automorphism group of $(A_0,\lambda_0)$, an inner form of $G = Sp_{2g}$ which
splits over $\mathbb Q_{p^2}$; this group was denoted by $G_{x_0}$ in Section
3. Denote by $L$ the completion of the maximal
unramified extension of $\mathbb Q_p$. We identify $G(L) = G'(L)$, and keep
track of the difference between the two groups by means of the two
different Frobenius actions. Denote the Frobenius on $G(L)$ giving rise to
the split form by $\sigma$, and the Frobenius giving rise to $G'$ by
$\sigma' = \mathop{\rm Int}(b) \circ \sigma$, for a suitable 
$b\in GSp_{2g}(L)$.

To make things completely explicit, we identify the chain of Dieudonn\'e
modules $M(A_i)$ of the $A_i$, $i=0,\dots, g$ (inside their common
isocrystal) with the standard lattice chain $\stdp_{-i}$, $i=0,\dots, g$
(inside $V = L^{2g}$), cf.~Sections \ref{gln}, \ref{unit_gp}. As above, we
identify $G$ with $Sp(V, \langle \cdot,\cdot \rangle_0)$, where $\langle
\cdot,\cdot \rangle_0$ is given by
\[
\begin{pmatrix}
  0 & \wt I_g \\
  -\wt I_g & 0
\end{pmatrix},\quad \wt I_g=\text{anti-diag(1,\dots,1)}.
\]
With the notation of Section \ref{subsec34}, we have $b=-g_\tau$.
We can (and do) choose the identification of $M(A_i)$ with $\Lambda_{-i}$
such that the Frobenius $F$ on $M(A_i)$ corresponds to $b\sigma$ with 
\[
b = \left( \begin{array}{cc} 0 & -I_g \\ pI_g & 0  \end{array}  \right).
\]
This is consistent with the setup in Section \ref{unit_gp}; see
(\ref{eq3.3}). We have $\sigma' = \mathop{\rm Int}(b) \circ \sigma$ with
this $b$.
The Iwahori subgroup associated with our chain is the
standard Iwahori subgroup $\Iw=\Iw'$ in $G(L)=G'(L)$.

Denote by 
$I = \{ 0, \dots, g \}$ the set of vertices of the extended
Dynkin diagram (of type $\tilde{C}_g$, see the figure). 
\[
\xymatrix@R=.1cm{
    0 & 1 & 2 & & g-2 & g-1 & g \\
    \bullet \ar@2{-}[r] &  \bullet \ar@{-}[r] &  \bullet \ar@{-}[r] &
    \cdots \ar@{-}[r] &  \bullet \ar@{-}[r] &  \bullet \ar@2{-}[r] &  \bullet
}
\]
The Galois group $\mathop{\rm Gal}(\mathbb Q_{p^2}/\mathbb Q_p)$ acts on
this set, considered as the extended Dynkin diagram of $G'$. Specifically,
the non-trivial element induces the map $i \mapsto g-i$ on $I$. For each
non-empty subset $J \subseteq I$ we have the parahoric subgroup $P_J\subset G(L)$ of
$G$, which we define as the subgroup generated by the Iwahori subgroup $\Iw$ and the
affine simple reflections $s_i$, $i\not\in J$. (Note that this notation is
not the usual one (where $P_J$ would be the subgroup generated by $\Iw$ and
$s_j$ for $j\in J$) --- in particular, for us $P_I = \Iw$.)
If the subset $J$ is Galois-stable, then $P_J$ is at the same time the
underlying set of a parahoric subgroup $P'_J$ of $G'$. We denote by $\mathbf P'_J$
the corresponding smooth group scheme in the sense of Bruhat-Tits theory,
and by $\overline{G}'_J$ the maximal reductive quotient of the special
fiber of $\mathbf P'_J$. (For instance, $\ol G'_{\{0,g\}}$ is the group denoted
by $\ol{\mathbf G}_0$ in Section \ref{unit_gp}, see Lemma \ref{31}.)

The Dynkin diagram of $\overline{G}'_J$ is obtained from the extended
Dynkin diagram $\tilde{C}_g$ by deleting the vertices in $J$. The group
$\overline{G}'_J$ is not split, but splits over $\mathbb F_{p^2}$. The
Frobenius of $\mathbb F_{p^2}$ over $\mathbb F_p$ is induced from the
Frobenius $\mathop{\rm Int}(b) \circ \sigma$ on $G'(L)$. In particular, it
acts on the Dynkin diagram by $i \mapsto g-i$ (cf.~\cite{tits:corvallis}).

The diagonal maximal torus in $G$ is $\sigma'$-stable, and hence can be
considered as a (non-split) maximal torus of $G'$ over $\mathbb Q_p$. Similarly, the standard
Borel subgroup $B \subseteq G(L)$ ``is'' a Borel subgroup of $G'$.
Especially, we can identify the Weyl groups of $G$ and $G'$ with respect to
these maximal tori. The fixed maximal torus of $G'$ gives rise to a maximal
torus of $\overline{G}'_J$, and the Weyl group $\overline{W}'_J$ of
$\overline{G}'_J$ with respect to this torus is isomorphic to the parabolic
subgroup $W_J\subset W_a$ corresponding to $J$.

Let $\overline{B}'_J \subset \overline{G}'_J$ be the image of the Iwahori
group $\Iw'=P'_I\subset G'$ in $\overline{G}'_J$; this is a Borel subgroup of
$\overline{G}'_J$. We have a commutative diagram
\[
\begin{CD}
P'_I \times P'_J/P'_I @>>> \overline{B}'_J \times
\overline{G}'_J/\overline{B}'_J \\
@VVV  @VVV \\
P'_J/P'_I @>\cong>> \overline{G}'_J/\overline{B}'_J
\end{CD}\]
which shows that there is a $1:1$ correspondence between the $P'_I$-orbits
in $P'_J/P'_I$ and the $\overline{B}'_J$-orbits in
$\overline{G}'_J/\overline{B}'_J$. This correspondence is compatible with our
identification of $W_J$ and $\overline{W}'_J$.

\begin{prop} \label{maxl_ss_kr}
Let $J \subseteq I$ be a non-empty Frobenius-stable subset.
Let $\pi_{J,I} \colon \mathcal A_I \rightarrow \mathcal A_J$ be the
projection, see Section \ref{sec:24}. We have an isomorphism
\[
\pi_{J,I}^{-1}((A_i)_{i\in J})\isomarrow \overline{G}'_J /
\overline{B}'_J
\]
of schemes over $k$,
where $(A_i)_i$ is the point in the minimal KR stratum fixed above.
\end{prop}

\begin{proof}
The space on the left hand side consists of all chains $(B_i)_i\in\mathcal
A_I$ with $B_j = A_j$ for $j \in J$ (and such that the isogenies $B_j
\rightarrow B_{j'}$ for $j,j'\in J$ coincide with the fixed isogenies $A_j
\rightarrow A_{j'}$.). Let $\tilde{J} = \{ \pm j + 2g k;\ j\in J,
k\in\mathbb Z \}$. We extend the chains $(A_i)_i$, $(B_i)_i$ by duality
so that they have index set $\mathbb Z$; then $A_i=B_i$ for all
$i\in\tilde{J}$.

Denote by $\omega_i \subset H^1_{DR}(B_i)$ the Hodge filtration. The Hodge
filtration of $A_i$ is just the image of $H^1_{DR}(A_{i+g})$ in
$H^1_{DR}(A_i)$. For $j\in J$, we obtain that $\omega_j$ is the image of
$H^1_{DR}(B_{j+g})$ in $H^1_{DR}(B_j)$.

Now let $S$ be a $k$-scheme. To an $S$-valued point $(B_i)_i\in
\pi_{J,I}^{-1}((A_i)_{i\in J})$ and elements $j_0 < j_1$ of $\tilde{J}$,
such that no $i$, $j_0<i<j_1$, lies in $\tilde{J}$, we can associate the
flags
\[
0 
\subsetneq
\alpha(\omega_{j_1-1}) \subsetneq \alpha(\omega_{j_1-2}) \subsetneq \cdots \subsetneq
\alpha(\omega_{j_0+1}) \subsetneq H^1_{DR}(A_{j_0+g}) / H^1_{DR}(A_{j_1+g})
= \ol\stdp_{-j_0-g}/\ol\stdp_{-j_1-g},
\]
where by abuse of notation, for every $i$, $j_0<i<j_1$, we denote by
$\alpha(\omega_{i})$ the image of $\omega_i$ in $\omega_{j_0} /
\omega_{j_1} = H^1_{DR}(A_{j_0+g}) / H^1_{DR}(A_{j_1+g})$.
Since the number of steps is equal to the dimension of the space on the
right hand side, and because the dimension difference at each step is at
most $1$, it must indeed be equal to $1$, which means that we have strict
inclusions at each step, as indicated above.
Taking into account the periodicity and the duality conditions, it is
clear that the collection of these flags is the same as an $S$-valued point
of $\overline{G}'_J / \overline{B}'_J$. In particular, we obtain a morphism
$\pi_{J,I}^{-1}((A_i)_{i\in J})\rightarrow \overline{G}'_J /
\overline{B}'_J$ (over $k$).

Now let $K\supseteq k$ be any perfect field. We want to show that
the morphism we constructed is bijective on $K$-valued points. We use the
description of $K$-valued points of the left hand side by \dieu theory.  We
have $H^1_{DR}(A_i) = M(A_i)/p (= \stdp_{-i}/p)$.  Given a flag in
$\overline{G}'_J / \overline{B}'_J$, we can lift it to chains
\[
M(A_{j_1 +g}) \subset \omega_{j_1-1} \subset \cdots \subset \omega_{j_0+1}
\subset M(A_{j_0+g})
\]
where $j_0, j_1\in J$ are
as above.  Now Verschiebung induces a bijective $\sigma^{-1}$-linear map
$M(A_{j_0})/M(A_{j_1}) \isomarrow M(A_{j_0+g})/M(A_{j_1+g})$. We set $M_i
:= V^{-1}\omega_i$ and obtain a chain
\[
M(A_{j_1}) \subset M_{j_1-1} \subset \cdots \subset M_{j_0+1}
\subset M(A_{j_0}).
\]
Because $FM(A_{j_0}) = VM(A_{j_0}) = M(A_{j_0-g}) \subset M(A_{j_1})$, the
$M_i$ are automatically stable under $F$ and $V$, and are the unique
chain of \dieu modules such that the images under Verschiebung are the
$\omega_i$. We obtain a unique chain of abelian varieties in 
$\pi_{J,I}^{-1}((A_i)_{i\in J})(K)$ which is mapped
to the point we started with in the flag variety $\overline{G}'_J /
\overline{B}'_J(K)$. This proves that we have a bijection on $K$-valued
points. In particular, the morphism is universally bijective and,
since it is proper, is a universal homeomorphism. It follows that both sides
have the same dimension.

Now consider the induced morphism on the tangent spaces, i.~e.~on
$K[\varepsilon]/\varepsilon^2$-valued points where the underlying
$K$-valued point is fixed. By the theory of Grothendieck and Messing, a
lift of the chain of abelian varieties over $K$ to $K[\varepsilon]$
corresponds to a lift of the Hodge filtration. This means that our morphism
is an isomorphism on the tangent spaces. It follows that the left hand side
is smooth and that the morphism is separable, and hence is in fact an
isomorphism.
\qed
\end{proof}

\begin{remark}
Note that at first one could think that the map which sends a chain
$(B_i)_i$ to the flag
\[
  0\subsetneq  \alpha(H^1_{DR}(B_{j_1-1})) \subsetneq \cdots \subsetneq 
   H^1_{DR}(B_{j_0})  / H^1_{DR}(B_{j_1}) = 
   H^1_{DR}(A_{j_0}) / H^1_{DR}(A_{j_1})
\]
gives the desired isomorphism.
This is again a bijection on $K$-valued points. But the theory of
Grothendieck and Messing shows that on the tangent spaces, this map induces
the zero map; it is a purely inseparable morphism.
\end{remark}

We now analyze the restriction of this isomorphism to the intersection with
a KR stratum.

\begin{thm} \label{sskr_are_dl}
Let $J \subseteq I$ be a non-empty Frobenius-stable subset.
Let $w \in W_J$. The isomorphism 
$\pi_{J,I}^{-1}((A_i)_{i\in J})\isomarrow \overline{G}'_J /
\overline{B}'_J$ restricts to an isomorphism
\[
\mathcal A_{w\tau} \cap \pi_{J,I}^{-1}((A_i)_{i\in J}) \isomarrow X(w^{-1}).
\]
\end{thm}

\begin{proof}
We can check this assertion on $k$-valued points, because the KR strata, as
well as the Deligne-Lusztig varieties, are
reduced.  Let $h \in \mathbf P'_J(k)$, and let $\dot{h}\in P'_J$ be a lift
of $h$. We can describe the image point of $h$ in $P'_J /I' \cong
\overline{G}'_J / \overline{B}'_J \cong \pi^{-1}_{J,I}((A_i)_{i\in _J})$ as
follows: it is the chain $(B_i)_i$ of abelian varieties with $B_i = A_i$
for $i\in J$, and with chain of Dieudonn\'e modules $V^{-1}\dot{h} 
\stdp_{-i-g}$, $i=0,\dots, g$ (where $\stdp_{-i}$ is the Dieudonn\'e
module of $A_i$ according to our normalization fixed above). We rewrite
this as
\[
V^{-1}\dot{h}\stdp_{-i-g} = \tau^{-1}\sigma(\dot{h})\tau \stdp_{-i} =
\sigma'(\dot{h}) \stdp_{-i}.
\]
The image of this \dieu module under Verschiebung is $\dot{h}\stdp_{-i-g}$
(in fact, this was the definition of the morphism
$\pi_{J,I}^{-1}((A_i)_{i\in J})\isomarrow \overline{G}'_J /
\overline{B}'_J$, so to say).

We have an isomorphism $\psi_i \colon H^1_{DR}(B_i) = \sigma'(\dot{h})
\stdp_{-i} / p \isomarrow \stdp_{-i} / p =
\overline{\stdp}_{-i}$ given by $\sigma'(h)^{-1}$, so 
the corresponding point in the local model (which is obtained as the
image under $\psi_\bullet$ of the reduction modulo $p$ of the image of $V$) 
is
$\sigma'(h)^{-1}h \overline{\stdp}_{-i-g} = \sigma'(h)^{-1}h \tau 
\overline{\stdp}_{-i}$ (note that
here $\sigma'$ is the Frobenius on the reduction $\mathbf P'_J$ over $k$,
and $\tau$ is the element in $G(k((t)))$ which induces the shift of the
lattice chain over $k[[t]]$). 

So the element $h \in \mathbf P'_J(k)$ gives rise to an element in the KR
stratum $\mathcal A_{w\tau}$ if and only if $\sigma'(h)^{-1}h\tau \in
\Iw w\Iw\tau$ (note that $\tau$ normalizes $\Iw$). Because of the correspondence
between Iwahori orbits in $P'_J/\Iw'$ and $\overline{B}'_J$-orbits in
$\overline{G}'_J/\overline{B}'_J$ discussed above, we can reformulate this
condition as $\sigma'(h)^{-1}h \in \overline{B}'_J w
\overline{B}'_J$, or equivalently as $h^{-1}\sigma'(h) \in \overline{B}'_J
w^{-1} \overline{B}'_J$, where we denote the image of $h$ in
$(\mathbf P'_J)^{\rm red} = \overline{G}'_J$ again by $h$. 
This proves our claim.

We can subsume this discussion in the following commutative diagram over $k =
\overline{\mathbb F}_p$ (we omit the subscript $k$) extending the local
model diagram (\ref{eq:29}).
\[ 
\xymatrix{
    \overline{G}'_J / \overline{B}'_J \ar[d]^{=} & (\mathbf P'_{J})^{\rm
    red} \ar[l] \ar[r] & \mathbf P_{J}^{\rm red} & \\
    \mathbf P'_{J} / I' \ar[d]^{\cong}& \mathbf P'_{J} \ar[u] \ar[l]
    \ar[d] \ar[r]^\beta & \mathbf P_{J} \ar[u]
    \ar[d]^{\cdot\tau}\\
    \pi^{-1}_{J,I}((A_i)_{i\in J}) \ar[d] &
    \widetilde{\pi^{-1}_{J,I}((A_i)_{i\in J})} \ar[l]
    \ar[r] \ar[d] &  (\mathbf P_J/I)\tau \ar[r] \ar[d] & W_J\tau \ar[d] \\
    \mathcal A_I & \widetilde{\mathcal A}_I \ar[l] \ar[r] &  M^{\rm
    loc} \ar[r] \ar[d] & {\rm Adm}(\mu) \ar[d] \\
     & & \Flag_{GSp_{2g}} \ar[r] & \wt W
}\]
The maps in the first row are just those induced from the second row. In
the second row, $\beta$ is the map
$\mathbf P'_{J}\rightarrow \mathbf P'_{J}$, $h \mapsto \sigma'(h)^{-1}h$,
followed by our identification $\mathbf P'_{J} = \mathbf P_{J}$. The fourth
row is essentially the local model diagram, and the third row is its
restriction to $\pi_{J,I}^{-1}(x_J)$ (i.~e.~we define
$\widetilde{\pi^{-1}_{J,I}(x_J)}$ as the inverse image of
$\pi_{J,I}^{-1}(x_J)$ in $\widetilde{\mathcal A}_I$.
\qed
\end{proof}

The following lemma will allow us to put all these pieces together in order
to obtain a description of the whole stratum $\mathcal A_w$.

\begin{lm}\label{degeneration_to_ss}
Let $J \subseteq I$ be a non-empty Frobenius-stable subset, and let $w\in W_J$. Let
$\mathcal A^0_{w\tau}$ be a connected component of the KR stratum $\mathcal
A_{w\tau}$.
Then the closure of $\mathcal A^0_{w\tau}$ in $\mathcal A_I$ meets the minimal KR
stratum.
\end{lm}

\begin{proof}
Let $z \in \mathcal A^0_{w\tau}$, and let $(B_i)_i$ be the corresponding chain of
abelian varieties. By assumption, we can identify the partial chain
$(M(B_i))_{i\in J}$ of \dieu modules for $i\in J$ with the chain
$(\stdp_{-i})_{i\in J}$, as \dieu modules, where Frobenius on $\stdp_{-i}$
is given by $b\sigma$ with $b$ as above.  Once we fix such an
identification, there exists a unique chain $(A_i)_i$ of abelian varieties
such that $A_i = B_i$ for $i\in J$ (compatibly with the isogenies between
these), and such that $M(A_i) = \stdp_{-i}$ for all $i$ (as \dieu modules).
In particular we can identify the chain $\omega(A_i) \subset H^1_{DR}(A_i)$
of Hodge filtrations with the chain $\ol\stdp_{-i-g}\subset\ol\stdp_{-i}$.
So the chain $(A_i)_i$ gives rise to a point in the minimal KR stratum,
and $z \in \pi_{J,I}^{-1}((A_i)_{i\in J})$. Since
$\pi_{J,I}^{-1}((A_i)_{i\in J})$ is isomorphic to a flag variety, where by the
preceding theorem the KR strata correspond to Deligne-Lusztig varieties, the
closure relations of KR strata correspond to the Bruhat order, and the lemma
follows.
\qed
\end{proof}

In fact more generally
we expect that whenever we take a connected component of a KR stratum, then
its closure meets the minimal KR stratum.
Altogether, we obtain the following description of supersingular KR strata:

\begin{cor} \label{desc_ss_kr_strata}
Let $J \subseteq I$ be a non-empty Frobenius-stable subset, and
let $w \in W_J$.
We have an isomorphism
\[
\mathcal A_{w\tau} \isomarrow \coprod_{x \in \pi_{J,I}(\mathcal A_\tau)}
X(w^{-1}).
\]
\end{cor}

\begin{proof}
We obtain this isomorphism by putting together all the isomorphisms of the
previous theorem. The lemma above implies that the right hand side indeed
is all of $\mathcal A_{w\tau}$.
\qed
\end{proof}

As a further consequence, we obtain that KR strata $\mathcal A_{w\tau}$
with $w$ as above are always quasi-affine, and are affine if $p \ge 2g$,
which is an upper bound for the Coxeter numbers of the groups
$\overline{G}'_{\{i,g-i\}}$ (see Section \ref{DL_affine}). We show in
\cite{goertz-yu:kreo} that all KR strata are quasi-affine.

\subsection{Number of connected components}

We compute the number of connected components of each $i$-superspecial
KR stratum. 

\begin{cor} \label{num_conn_comp}
Let $J \subsetneq I$ be a Frobenius-stable subset, and let $w \in W_J$.
We assume that $J$ is minimal with the property that it is Frobenius-stable
and $w\in W_J$. Then the number of connected components of $\mathcal
A_{w\tau}$ is
\[
\#\Lambda_{g,1,N} \cdot \#\overline{G}'_{\{0,g\}}/B'_{\{0,g\}}(\mathbb
F_p)
(\#\overline{G}'_J/\overline{B}'_J(\mathbb F_p))^{-1}.
\]
\end{cor}

Note that $\overline{G}'_{\{0,g\}}$ is the group denoted $\overline{\mathbf
G}_{0}$ in Section 3, where explicit formulas for the first two factors
were given. We will come back to making the whole formula explicit in
\cite{goertz-yu:mass}.

\begin{proof}
By Corollary \ref{desc_ss_kr_strata}, the number of connected
components is $\# \pi_{J,I}(\mathcal A_\tau)$, because the
Deligne-Lusztig variety $X(w^{-1})$ 
is connected by our assumption on $J$; see Corollary \ref{35}. Now
each fiber of the 
map $\mathcal A_\tau \rightarrow \pi_{J,I}(\mathcal A_\tau)$ can be
identified with $\overline{G}'_J/\overline{B}'_J(\mathbb F_p)$, as we
see from 
Theorem \ref{sskr_are_dl}. Since the term in the numerator
is $\# \mathcal A_\tau(k)$ by Cor.~\ref{35}, the corollary follows.
\qed
\end{proof}

\section{The unitary case}
\label{sec:07}
It is an obvious question whether the results about supersingular KR strata
generalize to other Shimura varieties of PEL type.  We are convinced that
by the same method, one obtains a geometric description of KR strata which
are entirely contained in the basic locus (let us call these the
\emph{basic KR strata}) in other cases, too, and we intend to come back to
this question in a future paper.  For the moment, we will restrict
ourselves to pointing out that nevertheless the Siegel case is particularly
well-adapted to this method. The reason is that in general, one would
expect that the basic KR strata make up only a very small part of the basic
locus. For an extreme case, let us consider the fake unitary case,
associated to a unitary group which splits over an unramified extension of
$\mathbb Q_p$. Let $(r,s)$ be the signature of this unitary group over
$\mathbb R$. The extended Dynkin diagram is a circle with $r+s$ vertices,
and Frobenius acts on it by a shift by $r$ steps (or, depending on the
setup, by a shift by $s$ steps). Now if $r$ and $s$ are coprime, then the
only non-empty Frobenius-stable subset of the set of vertices is the set of
all vertices. As a consequence, there are no parahoric subgroups as in
Section 6.1 except for the Iwahori subgroup itself, and the only KR stratum
we get by our method is the zero-dimensional one. Assuming that the
analogue of Theorem \ref{thm_ss_KR_strata} holds, this is the only
basic KR stratum.

\begin{thank}
This paper originated while the second named author was at the
Max-Planck-Institut 
f\"ur Mathematik in Bonn. He wants to thank the MPI for its kind
hospitality and good working conditions.
We gave talks about this subject at the \emph{Arbeitsgemeinschaft
Arithmetische Geometrie} at the Universit\"at Bonn, and would like to thank
the participants who made helpful comments. In particular we thank Michael
Rapoport for his steady interest and many inspiring suggestions and
questions, as well as for his remarks on a previous version of this text.
\end{thank}

\end{document}